\documentclass[a4paper] {article}
[12pt]

\usepackage[top=2.28cm, bottom=2.28cm, left=3.28cm, right=3.28cm]{geometry}

\makeatletter
\renewcommand*\l@section{\@dottedtocline{1}{1.5em}{2.3em}}
\makeatother

\usepackage{amsfonts}
\usepackage{amssymb}
\usepackage[T1]{fontenc}

\usepackage{tikz}
\usetikzlibrary{calc}

\usepackage{CJK}
\usepackage{amsmath}

\usepackage{amsfonts}
\usepackage{amssymb}
\usepackage{amsthm}
\usepackage{amssymb}
\usepackage{enumerate}
\usepackage[calc]{picture}
\usepackage[all,cmtip]{xy}

\usepackage[mathscr]{eucal}
\usepackage{eqlist}

\usepackage{color}
\usepackage{abstract}
\usepackage[T1]{fontenc}

\theoremstyle{plain}
\newtheorem{theorem}{Theorem}
\newtheorem{proposition}[theorem]{Proposition}
\newtheorem{lemma}[theorem]{Lemma}

\newtheorem{example}[theorem]{Example}
\newtheorem{corollary}[theorem]{Corollary}

\theoremstyle{definition}
\newtheorem{definition}{Definition}
\newtheorem{notation}{Notation}

\usepackage{etoolbox}
\newtheoremstyle{myrem}
 {3pt}
 {3pt}
 {\normalsize}
 { }
 {\itshape}
 {:}
 { }
 {}

 \theoremstyle{myrem}
 \newtheorem{remark}{Remark}
 \appto\remark{\leftskip\parindent}
 \appto\remark{\rightskip\parindent}

\numberwithin{equation}{section}
\numberwithin{theorem}{section}

\begin{document}

\begin{center}
{\Large {\textbf {
 Discrete  Differential Calculus on   Simplicial Complexes  and  Constrained  Homology
}}}
 \vspace{0.58cm}\\

Shiquan Ren

\smallskip

\begin{quote}
\begin{abstract}
Let $V$ be a finite set.   Let $\mathcal{K}$ be a  simplicial complex with its vertices in $V$.   In this paper,  we discuss some differential calculus on  $V$.   We   construct some  constrained  homology groups  of $\mathcal{K}$ by using the differential calculus on $V$.  Moreover,  we  define  a  independent hypergraph  to  be  the complement   of a  simplicial complex   in the complete hypergraph on $V$.    Let $\mathcal{L}$ be a   independent hypergraph with its vertices in $V$.    We   construct some  constrained  cohomology groups  of $\mathcal{L}$ by using the differential calculus on $V$.
\end{abstract}

\smallskip



\end{quote}

\end{center}

\smallskip

\section{Introduction}

Simplicial  complexes  play  an important and fundamental  role in algebraic topology.
 So far,  topologists have developed the homology and cohomology theory for simplicial complexes.  We refer to \cite[Chapter~1]{eat}  and  \cite[Section~2.1]{hatcher}  for a   systematic  introduction  to  the   simplicial homology theory.  We also   refer to \cite[Section~42, Chapter~5]{eat}  and \cite[Section~3.1 and Section~3.2]{hatcher}  for an introduction to  the  simplicial cohomology theory.    On  the other hand,  since 1950's,  topologists have developed the simplicial homotopy theory (for example,   we  may  refer  to  \cite{curtis, semi1,go, semi2, jiewu2}),   which has  been found to have significant  applications in various topics in algebraic and geometric topology (for example,  we refer to   \cite{cohen,jiewu1, wu5} for  some  of  such applications).     In  simplicial homotopy theory,  simplicial complexes  are   the   fundamental   models  for  simplicial sets.

 The notion of  hypergraphs  is  a higher  dimensional generalization of  the notion of graphs  (cf.  \cite{berge, parks}).   In a graph,  an edge consists of two vertices while in an  oriented hypergraph,  a oriented hyperedge is allowed   to be  consisted of $n$-vertices  for  any $n\geq 1$.  From a topological point of view,  an oriented  hypergraph can be obtained by deleting some non-maximal faces  in  an  oriented  simplicial complex   (cf. \cite{h1, parks})  while a oriented simplicial  complex  is a special  oriented hypergraph with no non-maximal faces missing.  The embedded homology of  hypergraphs was  introduced by  Stephane Bressan,  Jingyan  Li,  Shiquan  Ren  and Jie Wu     \cite{h1}.  The embedded homology of  oriented  hypergraphs was  proved to be  independent  on the choice of orientations  by Jelena  Grbi\'{c},  Jie  Wu,  Kelin Xia    and Guo-Wei Wei   \cite[Theorem 2.7]{bio1}.

   The  complete hypergraph  $\Delta[V]$  on  a finite set $V$  has its set of the hyperedges as  all the non-empty subsets of $V$ (cf.  Definition~\ref{def-4}).   A  simplicial complex with   all  of  its vertices in $V$ has its set of the simplices as a subset of $\Delta[V]$.   We  call  the complement of   the set of the simplices in $\Delta[V]$    a complement  hypergraph  (cf.  Definition~\ref{def-6}  and  Proposition~\ref{pr-x-2.1.1}).

Differential calculus  is an important tool in  (co)homology theory.    In  some textbooks   in  algebraic topology (for example,  \cite{botu,madsen}),     the methods  of differential calculus  have been  applied  to   the (co)homology theory of  differentiable manifolds and fibre bundles.  During  the  1990s,   A. Dimakis  and   F. M\"{u}ller-Hoissen  \cite{d1,d2,d3}   initiated the study of  discrete differential calculus  on discrete sets  with  a motivation from  theoretical physics.  During the 2010s,   based on the study of  \cite{d1,d2,d3},   Alexander Grigor'yan,  Yong  Lin  and   Shing-Tung Yau  \cite{lin1},   Alexander Grigor'yan,  Yong  Lin, Yuri Muranov and  Shing-Tung Yau \cite{lin2,lin3,lin4}  and   Alexander Grigor'yan,   Yuri Muranov and  Shing-Tung Yau \cite{lin5,lin6}  developed the  discrete differential calculus  methods on  discrete sets and applied  the methods  to the study of digraphs.

\smallskip

In this paper,  we apply  the method of the (discrete) differential calculus   and give  some  constrained   homology    for simplicial complexes as  well as constrained cohomology for independent hypergraphs.   The  constrained cohomology of  independent hypergraphs  that will be  introduced  in  this paper is in general different from the embedded homology of hypergraphs in \cite{h1} and the embedded cohomology of hypergraphs  in \cite{bio1}.

 Let $V$  be a finite set.   Let $\mathcal{K}$  be a simplicial complex whose  set of   vertices is  a   subset of  $V$.  Let $n\geq 0$.     Let $v_0v_1\ldots v_n$  be an $n$-simplex of $\mathcal{K}$.
The usual boundary operator   (cf.  \cite[p. 105]{hatcher},  \cite[p.  28]{eat})   is given by
\begin{eqnarray}\label{eq-usual-boundary}
\partial_n(v_0v_1\ldots v_n)=\sum_{i=0}^n(-1)^i   v_0\ldots \widehat{v_i}\ldots v_n.  \end{eqnarray}
We  generalize the usual boundary operator
  and define  a weighted boundary operator
\begin{eqnarray*}
\frac{\partial}{\partial v} (v_0v_1 \ldots v_n)=\sum_{i=0}^n (-1)^i  \delta(v,v_i) v_0\ldots \widehat{v_i}\ldots v_n
\end{eqnarray*}
with respect to any  fixed vertex $v\in V$.   Note that
\begin{eqnarray*}
\partial_n=\sum_{v\in V}\frac{\partial}{\partial v}.
\end{eqnarray*}
We  take the exterior algebra ${\rm Ext}_*(V)$  generated by  the $\frac{\partial}{\partial v}$'s  for all  $v\in V$.  We  prove in  Subsection~\ref{ss4.2}  that for any  $t\geq 0$ and any $\alpha\in {\rm Ext}_{2t+1}(V)$,  there is a constrained homology group of $\mathcal{K}$  with respect to $\alpha$.  Moreover,   we  prove in Theorem~\ref{th-4.1h} that  for any $s\geq 0$  and any  $\beta\in {\rm Ext}_{2s}(V)$,  the element $\beta$  induces a homomorphism between the corresponding constrained homology groups.

 We   point out  that  the constrained homology groups   which will be investigated in  Subsection~\ref{ss4.2}  are    generalizations of the weighted homology groups investigated by  Robert. J. MacG. Dawson \cite{1990}  and Chengyuan  Wu,   Shiquan  Ren,  Jie  Wu  and Kelin  Xia \cite{chengyuan1,chengyuan2,chengyuan3} for weighted simplicial complexes.  Let $f$  be a real function on $V$.  We take  $t=1$ and
 \begin{eqnarray*}
 \alpha= \sum_{v\in V} f(v) \frac{\partial}{\partial  v}
 \end{eqnarray*}
 in Definition~\ref{def-4.11},  Subsection~\ref{ss4.2}.
 Then the constrained homology groups of  the  simplicial complex $\mathcal{K}$  with respect to $\alpha$,  which will be   investigated in  Subsection~\ref{ss4.2},    give    the weighted homology groups of the weighted simplicial complex $(\mathcal{K},f)$   which  have been  investigated  in \cite{chengyuan1,chengyuan2,chengyuan3}.

On  the other hand,  Let $\mathcal{L}$  be a independent hypergraph whose set of    vertices is a  subset of  $V$.  For any $v\in V$,  we  consider the adjoint linear map  $dv$   of the element $\frac{\partial}{\partial v}$ in   ${\rm Ext}_*(V)$.  We   define   ${\rm Ext}^*(V)$   as the exterior algebra  generated  by the $dv$'s  for all $v\in V$.   We  prove in  Subsection~\ref{ss4.3} that for any  $t\geq 0$ and any $\omega\in {\rm Ext}^{2t+1}(V)$,  there is a constrained  cohomology group of $\mathcal{L}$  with respect to $\omega$.   Moreover,    we  prove in Theorem~\ref{th-4.2h}  that  for any $s\geq 0$  and any  $\mu\in {\rm Ext}^{2s}(V)$,  the element $\mu$  induces a homomorphism between the  constrained  cohomology groups.

\smallskip

The remaining part of this paper is organized as follows.  In  Section~\ref{s2},  we  introduce the definitions of  hypergraphs,  simplicial complexes and independent hypergraphs.  In  Section~\ref{s3},   as   a   preparation for Section~\ref{s4},  we  discuss some differential calculus  for paths on discrete sets.  In Section~\ref{s4},   we  define the constrained homology groups for simplicial complexes in Definition~\ref{def-4.11}   and define the constrained  cohomology groups for  independent hypergraphs in Definition~\ref{def-4.22}.  We prove Theorem~\ref{th-4.1h}  and Theorem~\ref{th-4.2h}.  Finally,  in Section~\ref{s5},  we  give some examples for Section~\ref{s4}.

\smallskip

\section{Hypergraphs,  Simplicial Complexes,  and Independent Hypergraphs}\label{s2}

Let  $V$  be  a   discrete  set whose elements are called {\it vertices}.   Let  $n\geq 0$  be a non-negative  integer. Let  $S_{n+1}$  be the symmetric group on $n$-letters.  Then  $S_{n+1}$ acts  on the set  of all the sequences $v_0v_1\ldots v_n$,   where  $v_0,v_1,\ldots,v_n\in V$,       by  permuting the orders of the vertices.

\begin{definition}
An  {\it oriented $n$-hyperedge} is an  equivalent class $[v_0,v_1,\ldots,v_n]$  where  the equivalence relation $\sim$  on the  set    $\{v_0v_1\ldots v_n\mid  v_0,v_1,\ldots,v_n\in V \}$  of  the sequences  is  given  by  $\sigma(v_0v_1\ldots v_n)\sim v_0v_1\ldots v_n$  iff  $\sigma\in S_n$  is an even  permutation.
\end{definition}

In the remaining part of this  paper,  suppose $V$  has a total order $\prec$.   

\begin{definition}\label{def1}
An {\it  $n$-hyperedge}   on $V$  is a sequence
\begin{eqnarray}\label{eq-2.1}
\sigma^{(n)}= v_0v_1\ldots v_n
\end{eqnarray}
where $v_0\prec  v_1\prec \cdots \prec  v_n$  are vertices in $V$.  For simplicity,  an    $n$-hyperedge  is also called    a {\it hyperedge}  and  $\sigma^{(n)}$ in (\ref{eq-2.1})  is also denoted as $\sigma$.
\end{definition}

\begin{remark}
By  Definition~\ref{def1},  a $0$-hyperedge  on $V$   is   just  a  single vertex  $v_0$  in $V$  and  a $1$-hyperedge on $V$  is just an edge $v_0v_1$  in the complete graph  on $V$.
\end{remark}

\begin{definition}\label{def-zxa}
The  {\it complete $n$-uniform  hypergraph}  $\Delta_n(V)$  on $V$  is the collection of all the possible $n$-hyperedges on $V$.   In  other words,  $\Delta[V]$ consists of all the   subsets  of $V$  with  $n$-vertices:
\begin{eqnarray*}
\Delta_n(V)=\{v_0v_1\ldots v_n \mid  v_0,  v_1, \cdots,  v_n\in  V {\rm~and~}   v_0\prec  v_1\prec \cdots \prec  v_n\}.
\end{eqnarray*}
\end{definition}

\begin{remark}
In particular,  let  $n=2$  in  Definition~\ref{def-zxa}.  Then  the   complete $2$-uniform  hypergraph   $\Delta_2(V)$  is just  the complete graph on $V$.
\end{remark}

\begin{definition}\label{def-2}
An {\it $n$-uniform hypergraph }   $\mathcal{H}^{(n)}$  on $V$  is a collection of some of the  $n$-hyperedges on $V$.  In  other words,   $\mathcal{H}^{(n)}$ consists of some of  the   subsets  of $V$  with $n$-vertices:
\begin{eqnarray*}
\mathcal{H}^{(n)}\subseteq \{v_0v_1\ldots v_n \mid  v_0,  v_1, \cdots,  v_n\in  V {\rm~and~}   v_0\prec  v_1\prec \cdots \prec  v_n\}.
\end{eqnarray*}
\end{definition}

\begin{definition}
A  {\it hypergraph}  on $V$ is a disjoint union
\begin{eqnarray}\label{eq-def0}
\mathcal{H}= \bigcup _{n\geq 0} \mathcal{H}^{(n)}
\end{eqnarray}
where   $\mathcal{H}^{(n)}$  is an  $n$-uniform hypergraph on $V$ for each $n\geq 0$.
\end{definition}

\begin{definition}\label{def-4}
 The  {\it complete hypergraph}  $\Delta[V]$  on $V$  is the collection of all the possible hyperedges on $V$.   In  other words,  $\Delta[V]$ consists of all the non-empty finite  subsets  of $V$.
 \end{definition}

 \begin{remark}
 It is direct that we  have a disjoint union
 \begin{eqnarray*}
 \Delta[V]= \bigcup_{n\geq  0} \Delta_n(V).
 \end{eqnarray*}
 \end{remark}

 \begin{definition}\label{def-xx}
 Let  $\mathcal{H}_1$ and  $\mathcal{H}_2$  be two  hypergraphs on $V$.  The  {\it complement}  of  $\mathcal{H}_1$  in $\mathcal{H}_2$  is defined to  be  a hypergraph $\mathcal{H}_2\setminus\mathcal{H}_1$ on $V$    by
 \begin{eqnarray*}
 \mathcal{H}_2\setminus\mathcal{H}_1=\{\sigma {\rm ~is~a~hyperedge~on~}V\mid \sigma\in \mathcal{H}_2{\rm~and~} \sigma\notin \mathcal{H}_1\}.
 \end{eqnarray*}
 \end{definition}

\begin{definition}\label{def-5}
A  {\it simplicial complex} (pl. {\it simplicial  complexes})  $\mathcal{K}$  on $V$  is a hypergraph on $V$ such that for any hyperedge $\sigma\in \mathcal{K}$  and any non-empty subset $\tau\subseteq\sigma$,  we always have $\tau\in \mathcal{K}$.   An      $n$-hyperedge in a simplicial complex  is   also   called an  {\it     $n$-simplex} (pl. {\it      $n$-simplices})  or  simply an      {\it       simplex} (pl. {\it       simplices}).
\end{definition}

\begin{definition}\label{def-6}
A  {\it independent hypergraph} 
    $\mathcal{L}$  on $V$  is a hypergraph on $V$ such that for any hyperedge $\sigma\in \mathcal{L}$  and any  hyperedge  $\tau$  on $V$  satisfying $\sigma\subseteq\tau$,  we always have $\tau\in \mathcal{L}$.
    \footnote{
    The reason that we use the term "independent hypergraph"  is as follows.  
By  Proposition~\ref{pr-x-2.1.1},  the set of    hyperedges of an  independent hypergraph  $\mathcal{L}$ on $V$
 is  the complement   of  the set of  simplices of a simplicial complex $\mathcal{K}$ on $V$  in the set of  hyperedges of the complete hypergraph $\Delta[V]$.  That  is, $\mathcal{K}=\Delta[V]\setminus \mathcal{L}$,  or equivalently,  $\mathcal{L}=\Delta[V]\setminus \mathcal{K}$.   If we regard each simplex $\sigma\in \mathcal{K}$ as a relation on $V$,  then the vertices of each hyperedge   $\sigma\in \mathcal{L}$ are independent  from  the relations in $\mathcal{K}$.   Since   $\mathcal{L}$ consists of all the hyperedges $\sigma\in \Delta[V]$ such that the  vertices  of each hyperedge are independent  from  the relations in $\mathcal{K}$,  we call $\mathcal{L}$  an  independent  hypergraph. 
 }
   
\end{definition}

 \begin{remark}
 From Definition~\ref{def-2},  Definition~\ref{def-5}  and Definition~\ref{def-6}, it  is direct that
\begin{itemize}
\item
for any $n\geq 1$,  an  $n$-uniform hypergraph  is not a simplicial complex;
\item
for any $n\leq \# V-1$  where $\# V$  is the cardinality of $V$ (here  $\# V$    can  be  either finite or infinite),   an $n$-uniform hypergraph  is not a independent hypergraph.
\end{itemize}
 \end{remark}

 \begin{remark}\label{remark-3}
  From Definition~\ref{def-4},  Definition~\ref{def-5}  and Definition~\ref{def-6}, it  is direct that
the complete hypergraph $\Delta[V]$  is a simplicial  complex  on $V$  and also  a independent hypergraph  on $V$.
 \end{remark}

 \begin{proposition}\label{pr-x-2.1.1}
Let $\Delta[V]$  be the complete hypergraph  on $V$.  Let $\mathcal{K}$  be a simplicial complex on $V$.   Let $\mathcal{L}$  be a independent hypergraph on $V$.  Then  both the followings are satisfied:
 \begin{enumerate}[(i).]
 \item
  $\Delta[V]\setminus\mathcal{K}$  is a independent hypergraph on $V$;
 \item
   $\Delta[V]\setminus\mathcal{L}$  is a simplicial complex on $V$.
 \end{enumerate}
 \end{proposition}
 \begin{proof}
 (i).  Let $\sigma\in\Delta[V]\setminus\mathcal{K}$.   Let  $\tau$  be a hyperedge on $V$  such that $\sigma\subseteq\tau$.  In order to prove that   $\Delta[V]\setminus\mathcal{K}$  is a independent hypergraph,  it suffices to prove $\tau\in \Delta[V]\setminus\mathcal{K}$.  Suppose to the contrary,  $\tau\notin \Delta[V]\setminus\mathcal{K}$.  Then $\tau\in \mathcal{K}$.  Since $\mathcal{K}$  is a simplicial complex and $\sigma\subseteq\tau$,  we  have $\sigma\in \mathcal{K}$.  This contradicts  $\sigma\in\Delta[V]\setminus\mathcal{K}$.  Therefore,  $\tau\in \Delta[V]\setminus\mathcal{K}$,  which implies that  $\Delta[V]\setminus\mathcal{K}$  is a independent hypergraph.

 (ii).   Let $\sigma\in  \Delta[V]\setminus\mathcal{L}$.  Let  $\tau$  be a hyperedge on $V$  such that $\tau\subseteq\sigma$.  In order to prove that   $\Delta[V]\setminus\mathcal{L}$  is a  simplicial complex,  it suffices to prove $\tau\in \Delta[V]\setminus\mathcal{L}$.  Suppose to the contrary,  $\tau\notin \Delta[V]\setminus\mathcal{L}$.  Then $\tau\in \mathcal{L}$.  Since $\mathcal{L}$  is a  independent hypergraph and $\tau\subseteq\sigma$,  we  have $\sigma\in \mathcal{L}$.  This contradicts  $\sigma\in\Delta[V]\setminus\mathcal{L}$.   Therefore,  $\tau\in \Delta[V]\setminus\mathcal{L}$,  which implies that  $\Delta[V]\setminus\mathcal{L}$  is a simplicial  complex.
 \end{proof}

 \smallskip

 \begin{example}Consider the set $V=\{v_0,v_1,v_2,v_3,v_4,v_5\}$.  Then
 \begin{enumerate}[(i).]
 \item
 $\sigma^{(3)}=v_0v_2v_4v_5$  is $3$-hyperedge on $V$;

 \item
 $\mathcal{H}^{(2)}=\{v_0v_2v_3, v_1v_2v_3,  v_1v_3v_5,  v_2v_4v_5\}$   is a $2$-uniform  hypergraph on $V$;

 \item
 $\mathcal{H}=\{v_0,  v_0v_1, v_4v_5,  v_0v_1v_2, v_2v_3v_4v_5\}$  is a hypergraph on $V$;

 \item
 $\Delta[V]=\{v_i\mid  0\leq i\leq 5\}\cup \{v_iv_j\mid  0\leq i<j\leq 5\} \cup \{v_iv_jv_k\mid  0\leq i<j<k\leq 5\}\cup \{v_iv_jv_kv_l\mid  0\leq i<j<k<l\leq 5\}\cup \{v_iv_jv_kv_lv_s\mid  0\leq i<j<k<l<s\leq 5\} \cup\{v_0v_1v_2v_3v_4v_5\}$;

 \item
 $\mathcal{K}=\{v_0, v_0v_1, v_0v_2, v_1v_2, v_0v_1v_2\}$  is  a simplicial  complex  on $V$;

 \item
 $\mathcal{L}=\{v_0v_1v_2 v_4,  v_0v_1v_2v_3v_5,    v_0v_1v_2v_3v_4, v_0v_1v_2v_4v_5, v_0v_1v_2v_3v_4v_5\}$   is a independent hypergraph  on  $V$.

 \end{enumerate}
 \end{example}

 \begin{example}
 Consider  the set  $V=\mathbb{Z}$  of all the integers.  Then
 \begin{enumerate}[(i).]
 \item
 for any $p\in\mathbb{Z}$ and any $q\geq 0$,    the sequence $p(p+1)\ldots (p+q)$  of  subsequent integers  is a $q$-hyperedge on $V$;

 \item
 for  any $q\geq 0$,  the collection  $\mathcal{H}^{(q)}=\{p(p+1)\ldots (p+q)\mid  p\equiv 1 ~({\rm mod~} 3)\}$  of sequences  of  subsequent integers  is a $q$-uniform  hypergraph on $V$;

 \item
 the collection $\mathcal{H} =\{p(p+1)\ldots (p+q)\mid  p\equiv 1 ~({\rm mod~} 3){\rm~and~} 2\leq q\leq 5\}$  of sequences  of  subsequent integers  is a hypergraph on $V$;

 \item
 $\Delta[V]=\{ i_0\in \mathbb{Z}\} \cup  \{i_0i_1\in  \mathbb{Z}\times \mathbb{Z}\mid  i_0<i_1\}\cup   \{i_0i_1i_2\in  \mathbb{Z}\times \mathbb{Z}\times\mathbb{Z}\mid  i_0<i_1<i_2\}\cup\cdots$;

 \item
 the collection $\mathcal{K}=\{p(p+1)\ldots (p+q)\mid p\in\mathbb{Z} {\rm ~and~}  0\leq q\leq 5\}$   of  sequences  of  subsequent integers  is a simplicial complex on $V$;

 \item
  the collection $\mathcal{L}=\{p(p+1)\ldots (p+q)\mid  p\in\mathbb{Z} {\rm ~and~}  q>5 \}$   of  sequences  is a independent hypergraph on $V$.

 \end{enumerate}
 \end{example}

 \smallskip

 \section{Differential Calculus for Paths on Discrete Sets}\label{s3}

In  this section,  we review the definitions of the paths and the elementary paths  on  a discrete set  (cf.  \cite{lin2}).   By applying some  discrete differential  calculus,   we  construct certain chain complexes   and  co-chain  complexes for  the space of paths on a discrete set.

\smallskip

\subsection{Paths on Discrete Sets}

Throughout this section,  we  let  $V$  be a  discrete set.
Let $n\geq 0$  be a non-negative integer.

\begin{definition}(cf.  \cite[Definition~2.1]{lin2}).
An {\it elementary $n$-path}  on $V$  is an ordered sequence $v_0v_1\ldots v_n$ of $n+1$ vertices in  $V$.  Here   for any integers $0\leq i<j\leq n$,   we do not require $v_i\prec v_j$, $v_j\prec v_i$ or  $v_i\neq v_j$.
\end{definition}

\begin{definition}(cf.  \cite[Definition~2.2]{lin2}).
  A formal linear combination of elementary $n$-paths on $V$ with coefficients in the real numbers  $\mathbb{R}$  is called an {\it $n$-path}  on $V$.
  \end{definition}

  \begin{notation}(cf.  \cite[Subsection~2.1]{lin2}).
  Denote by $\Lambda_n(V)$  the vector space of all $n$-paths.   Then any element in $\Lambda_n(V)$   is of the form
\begin{eqnarray*}
\sum_{v_0,v_1,\ldots,v_n\in V}{r_{v_0v_1\ldots v_n}v_0v_1\ldots v_n},~~~r_{v_0v_1\ldots v_n}\in \mathbb{R}.
\end{eqnarray*}
\end{notation}
\begin{notation}
Letting $n$ run over all non-negative integers,  we have a graded vector space
\begin{eqnarray*}
\Lambda_*(V)=\bigoplus_{n=0}^\infty \Lambda_n(V).
\end{eqnarray*}
\end{notation}

\begin{notation}
 For each $n\geq  0$,  we have a canonical  inner product
 \begin{eqnarray*}
\langle~,~\rangle:~~~  \Lambda_n(V)\times \Lambda_n(V)\longrightarrow \mathbb{R}
\end{eqnarray*}
on $\Lambda_n(V)$  by
\begin{eqnarray}\label{eq-1.2}
\langle u_0u_1\ldots u_n,v_0v_1\ldots v_n\rangle =\prod_{i=0}^n  \delta(u_i,v_i).
\end{eqnarray}
\end{notation}

\begin{remark}
It follows from (\ref{eq-1.2}) that
\begin{itemize}
\item
  if $u_0u_1\ldots u_n$  and $v_0v_1\ldots v_n$  are identically the same  elementary $n$-path,  then
\begin{eqnarray*}
\langle u_0u_1\ldots u_n,v_0v_1\ldots v_n\rangle=1;
\end{eqnarray*}
\item
if $u_0u_1\ldots u_n$  and $v_0v_1\ldots v_n$  are not the same  elementary $n$-path,  then
  \begin{eqnarray*}
  \langle u_0u_1\ldots u_n,v_0v_1\ldots v_n\rangle=0.
  \end{eqnarray*}
    \end{itemize}
    \end{remark}

\smallskip

\subsection{Partial Derivatives on Path Spaces}

\begin{definition} \label{def-13}
For any $v\in V$,  we define the {\it partial derivative} of $\Lambda_*(V)$ with respect to $v$ to be a sequence of   linear maps
\begin{eqnarray*}
\frac{\partial}{\partial v}: ~~~ \Lambda_n(V)\longrightarrow \Lambda_{n-1}(V),~~~
n\geq 0
\end{eqnarray*}
by letting
\begin{eqnarray}\label{eq-1.1}
\frac{\partial}{\partial v} (v_0v_1 \ldots v_n)=\sum_{i=0}^n (-1)^i  \delta(v,v_i) v_0\ldots \widehat{v_i}\ldots v_n.
\end{eqnarray}
 Here  in (\ref{eq-1.1}),   for  any  vertices  $u,v\in V$,  we  use the notation  $\delta(u,v)=1$  if $u=v$  and  $\delta(u,v)=0$  if  $u\neq v$.   We extend (\ref{eq-1.1})   linearly  over $\mathbb{R}$.
 \end{definition}

  \begin{remark}
By Definition~\ref{def-13},  for any distinct vertices  $v_0,v_1,\ldots,v_n$  in $V$  we  have the followings:
\begin{itemize}
\item
if $v_i=v$ for some $0\leq i\leq n$,  then
\begin{eqnarray*}
\frac{\partial}{\partial v} (v_0 v_1\ldots v_n)=(-1)^i   v_0\ldots \widehat{v_i}\ldots v_n;
\end{eqnarray*}
\item
if  $v_i\neq v$  for any $0\leq i\leq n$,   then
\begin{eqnarray*}
\frac{\partial}{\partial v} (v_0v_1,\ldots v_n)=0.
\end{eqnarray*}
\end{itemize}
\end{remark}

\begin{lemma}(\cite[Lemma~2.7]{sid}).  \label{le-0.0}
For any $u,v\in V$,  we  have
\begin{eqnarray}\label{eq-1.0}
\frac{\partial}{\partial u}\circ\frac{\partial}{\partial v}=-\frac{\partial}{\partial v}\circ\frac{\partial}{\partial u}.
\end{eqnarray}
\end{lemma}

\begin{proof}
Since both $\frac{\partial}{\partial u}$  and $\frac{\partial}{\partial v}$  are  linear,  it follows that both  $\frac{\partial}{\partial u}\circ\frac{\partial}{\partial v}$ and $\frac{\partial}{\partial v}\circ\frac{\partial}{\partial u}$
 are  linear as well.  Hence in order to prove the identity (\ref{eq-1.0}) as linear maps from $\Lambda_n(V;R)$ to $\Lambda_{n-1}(V;R)$,   we only need to verify the identity (\ref{eq-1.0}) on an elementary $n$-path $v_0v_1\ldots v_n$.   By  the definition (\ref{eq-1.1}),  we have
\begin{eqnarray*}
\frac{\partial}{\partial u}\circ\frac{\partial}{\partial v}(v_0v_1\ldots v_n)&=&\frac{\partial}{\partial u}\Big(\sum_{j=0}^n (-1)^j\delta(v,v_j)v_0 \ldots \widehat{v_j} \ldots v_n\Big)\\
&=&\sum_{j=0}^n (-1)^j\delta(v,v_j)\frac{\partial}{\partial u}(v_0 \ldots \widehat{v_j} \ldots v_n)\\
&=&\sum_{j=0}^n (-1)^j\delta(v,v_j)\sum_{i=0}^{j-1} (-1)^i \delta(u,v_i)(v_0 \ldots \widehat{v_i}\ldots \widehat{v_j} \ldots v_n)\\
&& +\sum_{j=0}^n(-1)^j\delta(v,v_j)\sum_{i=j+1}^{n} (-1)^{i-1} \delta(u,v_i)(v_0 \ldots \widehat{v_i}\ldots \widehat{v_j} \ldots v_n)\\
&=&\sum_{0\leq i<j\leq n}(-1)^{i+j}\delta(u,v_i)\delta(v,v_j)(v_0 \ldots \widehat{v_i}\ldots \widehat{v_j} \ldots v_n)\\
&&+ \sum_{0\leq j<i\leq n}(-1)^{i+j-1}\delta(u,v_i)\delta(v,v_j)(v_0 \ldots \widehat{v_j}\ldots \widehat{v_i} \ldots v_n).
\end{eqnarray*}
Similarly,
\begin{eqnarray*}
\frac{\partial}{\partial v}\circ\frac{\partial}{\partial u}(v_0v_1\ldots v_n)&=&
\sum_{0\leq j<i\leq n}(-1)^{i+j}\delta(u,v_i)\delta(v,v_j)(v_0 \ldots \widehat{v_i}\ldots \widehat{v_j} \ldots v_n)\\
&&+ \sum_{0\leq i<j\leq n}(-1)^{i+j-1}\delta(u,v_i)\delta(v,v_j)(v_0 \ldots \widehat{v_j}\ldots \widehat{v_i} \ldots v_n).
\end{eqnarray*}
Therefore,  for any elementary $n$-path $v_0v_1\ldots v_n$  on $V$,  we have
\begin{eqnarray*}
\frac{\partial}{\partial u}\circ\frac{\partial}{\partial v}(v_0v_1\ldots v_n)+\frac{\partial}{\partial v}\circ\frac{\partial}{\partial u}(v_0v_1\ldots v_n)=0.
\end{eqnarray*}
Consequently, by the   linear property of $\frac{\partial}{\partial u}\circ\frac{\partial}{\partial v}$ and $\frac{\partial}{\partial v}\circ\frac{\partial}{\partial u}$,  we   obtain (\ref{eq-1.0}).
\end{proof}

\begin{notation}
We denote $\frac{\partial}{\partial v}\circ\frac{\partial}{\partial u}$ as $\frac{\partial}{\partial v}\wedge \frac{\partial}{\partial u}$ for any $u,v\in V$.
\end{notation}

\begin{definition}
We consider the exterior algebra
\begin{eqnarray*}
{\rm Ext}_*(V)= \bigwedge\Big(\frac{\partial}{\partial v}\mid v\in V\Big)
\end{eqnarray*}
and call it the {\it differential algebra} on $V$.
\end{definition}

We have the following observations:
\begin{itemize}
\item
The differential algebra  ${\rm Ext}_*(V)$  is a direct sum
\begin{eqnarray*}
{\rm Ext}_*(V)=\bigoplus_{k=0}^\infty {\rm Ext}_k(V);
\end{eqnarray*}
\item
 ${\rm Ext}_0(V)=\mathbb{R}$ while   for each $k\geq 1$,  the space ${\rm Ext}_k(V)$  is the vector space spanned by all the following elements
\begin{eqnarray*}
\frac{\partial}{\partial v_1}\wedge \frac{\partial}{\partial v_2}\wedge \cdots \wedge\frac{\partial}{\partial v_k},~~~ v_1,v_2,\ldots, v_k\in V
\end{eqnarray*}
modulo the relation
\begin{eqnarray*}
\frac{\partial}{\partial v_1}\wedge \cdots \wedge \frac{\partial}{\partial v_i}\wedge \frac{\partial}{\partial v_{i+1}}\wedge \cdots  \wedge\frac{\partial}{\partial v_k}=-\frac{\partial}{\partial v_1}\wedge \cdots \wedge \frac{\partial}{\partial v_{i+1}}\wedge \frac{\partial}{\partial v_{i}}\wedge \cdots  \wedge\frac{\partial}{\partial v_k}
\end{eqnarray*}
for any $1\leq i\leq k-1$;
\item
  The exterior product
\begin{eqnarray*}
\wedge:~~~  {\rm Ext}_k(V)\times {\rm Ext}_l(V)\longrightarrow {\rm Ext}_{k+l}(V), ~~~ k,l\geq 1,
\end{eqnarray*}
 is   the composition of linear maps.  It   is  given by
\begin{eqnarray*}
(\frac{\partial}{\partial v_1}\wedge \cdots \wedge\frac{\partial}{\partial v_k})\wedge (\frac{\partial}{\partial u_1}\wedge \cdots \wedge\frac{\partial}{\partial u_l})=\frac{\partial}{\partial v_1}\wedge \cdots \wedge\frac{\partial}{\partial v_k}\wedge \frac{\partial}{\partial u_1}\wedge \cdots \wedge\frac{\partial}{\partial u_l}
\end{eqnarray*}
which  extends bi-linearly over $\mathbb{R}$.

 \item
  For any $k\geq 1$  and any  $\alpha \in {\rm Ext}_k(V)$,  we have     that  $\alpha$ gives a sequence of  linear maps
\begin{eqnarray}\label{eq-2.cc}
\alpha_n:~~~ \Lambda_n(V)\longrightarrow \Lambda_{n-k}(V),  ~~~ n\geq 0.
\end{eqnarray}
Here we adopt the notation that $\Lambda_{-n}(V)=0$  for any $n\geq 0$.   Precisely,  if  we write
\begin{eqnarray*}
\alpha=\sum_{v_1,v_2,\ldots, v_k\in V} r_{v_1,v_2,\ldots, v_k}\frac{\partial}{\partial v_1}\wedge \frac{\partial}{\partial v_2}\wedge \cdots \wedge \frac{\partial}{\partial v_k},~~~ r_{v_1,v_2,\ldots, v_k}\in\mathbb{R},
\end{eqnarray*}
then for any elementary $n$-path $u_0u_1\ldots u_n$  on $V$ with  $n\geq k$,  we have
 \footnote[2]{The   expression of $\alpha(u_0u_1\ldots u_n)$ follows from the following two  observations:
\begin{enumerate}[(i).]
\item
for any    $0\leq i_1<i_2<\ldots < i_k \leq n$,  by  applying  (\ref{eq-1.1})  for $k$-times,  we  have
\begin{eqnarray*}
\frac{\partial}{\partial u_{i_1}}\wedge \frac{\partial}{\partial u_{i_2}}\wedge \cdots \wedge \frac{\partial}{\partial u_{i_k}}(u_0u_1\ldots u_n)=(-1)^{i_1+i_2+\cdots+i_k} u_0 \ldots \widehat{u_{i_1}}\ldots  \widehat{u_{i_2}}\ldots  \widehat{u_{i_k}}\ldots u_n;
\end{eqnarray*}
\item
for any $\sigma\in S_k$,  by applying (\ref{eq-1.0})   iteratively,  we  have
\begin{eqnarray*}
\frac{\partial}{\partial u_{i_{\sigma(1)}}}\wedge \frac{\partial}{\partial u_{i_{\sigma(2)}}}\wedge \cdots \wedge \frac{\partial}{\partial u_{i_{\sigma(k)}}}= {\rm sgn}(\sigma)\frac{\partial}{\partial u_{i_1}}\wedge \frac{\partial}{\partial u_{i_2}}\wedge \cdots \wedge \frac{\partial}{\partial u_{i_k}}.
\end{eqnarray*}
\end{enumerate}
}
\begin{eqnarray*}
\alpha(u_0u_1\ldots u_n)&=&\sum_{v_1,v_2,\ldots, v_k\in V} r_{v_1,v_2,\ldots, v_k}\frac{\partial}{\partial v_1}\wedge \frac{\partial}{\partial v_2}\wedge \cdots \wedge \frac{\partial}{\partial v_k}(u_0u_1\ldots u_n) \\
&=&\sum_{0\leq i_1<i_2<\ldots <i_k\leq n ~}\sum_{\sigma\in S_k} r_{u_{i_{\sigma(1)}}, u_{i_{\sigma(2)}}, \ldots, u_{i_{\sigma(k)}}}  {\rm sgn}(\sigma)\\
&&{\rm ~~~~~~~~~~~~~~~~~~~~~~~~~~}(-1)^{i_1+i_2+\cdots+i_k} u_0 \ldots \widehat{u_{i_1}}\ldots  \widehat{u_{i_2}}\ldots  \widehat{u_{i_k}}\ldots u_n.
\end{eqnarray*}
Here $S_k$  is the permutation group on $k$-letters   and   for any permutation $\sigma\in S_k$,  we use ${\rm sgn}(\sigma)$  to denote  the signature of $\sigma$.
\end{itemize}

\smallskip

\subsection{Partial Differentiations on Path Spaces}

      \begin{definition}
   For any $v\in V$,  we define the {\it partial differentiation}  $dv$  with respect to $v$  to be a sequence of  linear maps
\begin{eqnarray*}
d v:  ~~~  \Lambda_{n}(V)\longrightarrow \Lambda_{n+1}(V), ~~~n\geq 0,
\end{eqnarray*}
such that $dv$  is the adjoint linear map of $\frac{\partial}{\partial v}$  for each $n\geq 0$.  Precisely,   for any   $n\geq 0$,  any $\xi\in \Lambda_n(V)$,   and any $\eta\in \Lambda_{n+1}(V)$,   we  have
\begin{eqnarray}\label{eq-2.in}
\Big\langle  \frac{\partial}{\partial v}(\eta),\xi\Big\rangle =\langle \eta, dv(\xi)\rangle.
\end{eqnarray}
\end{definition}

The next lemma gives an explicit formula for $dv$.

\begin{lemma}(\cite[Lemma~2.10]{sid}).
 \label{le-x}
For any $n\geq 1$,  any $v\in V$,  and any elementary $(n-1)$-path $u_0u_1\ldots u_{n-1}$  on $V$,  we  have
\begin{eqnarray} \label{eq-2.x5}
dv (u_0u_1\ldots u_{n-1})=\sum_{i=0}^n (-1)^i u_0u_1\ldots u_{i-1} v u_i u_{i+1}\ldots u_{n-1}.
\end{eqnarray}
\end{lemma}
\begin{proof}
In (\ref{eq-2.in}),  we take $\eta$  to be an elementary $n$-path $v_0v_1\ldots v_n\in \Lambda_n(V)$  and  take  $\xi$  to be an elementary $(n-1)$-path $u_0u_1\ldots u_{n-1}\in \Lambda_{n-1}(V)$.  Then
\begin{eqnarray*}
\langle v_0v_1\ldots v_n,  dv (u_0u_1\ldots u_{n-1})\rangle &=&  \langle \frac{\partial}{\partial v}(v_0v_1\ldots v_n),   u_0u_1\ldots u_{n-1} \rangle\\
&=&\langle \sum _{i=0}^n (-1)^i \delta(v,v_i)v_0\ldots \widehat{v_i}\ldots v_n, u_0u_1\ldots u_{n-1}\rangle\\
&=&\sum _{i=0}^n (-1)^i \delta(v,v_i)\prod_{j=0}^{i-1}\delta(v_j,u_j) \prod _{j=i}^{n-1}\delta(v_{j+1},u_{j}).
\end{eqnarray*}
Consequently,  we  have
\begin{eqnarray}
dv (u_0u_1\ldots u_{n-1})&=&\sum_{v_0,v_1,\ldots, v_n\in V}  \langle v_0v_1\ldots v_n, dv (u_0u_1\ldots u_{n-1}) \rangle  v_0v_1\ldots v_n\nonumber\\
&=&\sum_{v_0,v_1,\ldots, v_n\in V}\Big(\sum _{i=0}^n (-1)^i \delta(v,v_i)\prod_{j=0}^{i-1}\delta(v_j,u_j) \prod _{j=i}^{n-1}\delta(v_{j+1},u_{j})\Big) v_0v_1\ldots v_n\nonumber\\
&=&\sum_{i=0}^n (-1)^i \Big( \sum_{v_0,v_1,\ldots, v_n\in V} \delta(v,v_i)\prod_{j=0}^{i-1}\delta(v_j,u_j) \prod _{j=i}^{n-1}\delta(v_{j+1},u_{j})\Big)v_0v_1\ldots v_n\nonumber\\
&=&\sum_{i=0}^n (-1)^i u_0u_1\ldots u_{i-1} v u_i u_{i+1}\ldots u_{n-1}.
\end{eqnarray}
We obtain (\ref{eq-2.x5}).
\end{proof}

The  next corollary gives the case $n=1$  in Lemma~\ref{le-x}.

\begin{corollary}
For any $u,v\in  V$  we  have $dv(u)=vu-uv$.
\qed
\end{corollary}

Similar with  the proof of Lemma~\ref{le-0.0},    it is direct to verify  the next lemma.

\begin{lemma}(\cite[Lemma~2.7]{sid}).
 \label{le-2.xy}
For any $u,v\in V$   we   have
\begin{eqnarray}\label{eq-2.88}
du \circ  dv=- dv\circ  du.
\end{eqnarray}
\end{lemma}

\begin{proof}
\footnote[3]{An  alternative proof for Lemma~\ref{le-2.xy} follows from Lemma~\ref{le-0.0}  directly:  Let $u,v\in V$.  For any  any  $n\geq 0$,  any $\xi\in \Lambda_n(V)$    and any $\eta\in \Lambda_{n+2}(V)$,  we  have
\begin{eqnarray*}
&&  \langle    \eta , du\wedge  dv(\xi)\rangle= \Big\langle   \frac{\partial}{\partial u} (\eta), dv(\xi)\Big\rangle=\Big\langle \frac{\partial}{\partial v}\wedge \frac{\partial}{\partial u} (\eta),\xi\Big\rangle\\
&=&-\Big\langle \frac{\partial}{\partial u}\wedge \frac{\partial}{\partial v} (\eta),\xi\Big\rangle =-\Big\langle  \frac{\partial}{\partial v} (\eta),  du(\xi)\Big\rangle = -\langle \eta,  dv \wedge du (\xi)\rangle.
\end{eqnarray*}
    This implies (\ref{eq-2.88}).   Nevertheless,  the proof for Lemma~\ref{le-2.xy}  in the main-body  consolidates (\ref{eq-2.x5})  in Lemma~\ref{le-x}.
 }
For any $n\geq 0$ and any elementary $n$-path $v_0v_1\ldots v_n\in \Lambda_n(V)$,   by  (\ref{eq-2.x5}),   we have
\begin{eqnarray*}
du \circ  dv(v_0v_1\ldots v_n)&=&du\Big(\sum_{i=0}^{n+1} (-1)^i v_0\ldots v_{i-1} v v_i\ldots v_n\Big)\\
&=&\sum_{i=0}^{n+1} (-1)^i du(v_0\ldots v_{i-1} v v_i\ldots v_n)\\
&=&\sum_{i=0}^{n+1} (-1)^i \Big(\sum_{j=0}^{i-1}(-1)^j v_0\ldots v_{j-1}  u v_j \ldots v_{i-1} v  v_i \ldots v_n\\
&&+(-1)^i v_0\ldots v_{i-1}uvv_i\ldots v_n + (-1)^{i+1}v_0\ldots v_{i-1} v u v_i\ldots v_n\\
&&+  \sum_{j=i+1}^{n+1} (-1)^{j+1} v_0\ldots v_{i-1} v v_i \ldots v_{j-1} u v_j\ldots v_n\Big)
\end{eqnarray*}
while
\begin{eqnarray*}
dv \circ  du(v_0v_1\ldots v_n)&=&\sum_{i=0}^{n+1}(-1)^i \Big(\sum_{j=0}^{i-1}(-1)^j v_0\ldots v_{j-1}  v v_j \ldots v_{i-1} u  v_i \ldots v_n\\
&&+(-1)^i v_0\ldots v_{i-1}vuv_i\ldots v_n + (-1)^{i+1}v_0\ldots v_{i-1} u v v_i\ldots v_n\\
&&+  \sum_{j=i+1}^{n+1} (-1)^{j+1} v_0\ldots v_{i-1} u v_i \ldots v_{j-1} v v_j\ldots v_n\Big).
\end{eqnarray*}
Thus
\begin{eqnarray*}
du \circ dv(v_0v_1\ldots v_n)=
-dv \circ  du(v_0v_1\ldots v_n)
\end{eqnarray*}
 for any $n\geq 0$ and any elementary $n$-path $v_0v_1\ldots v_n\in \Lambda_n(V)$.   Consequently,  we obtain     (\ref{eq-2.88}).
 \end{proof}

\begin{notation}
We  denote  $d u\circ dv $  as $du \wedge dv$  for any $u,v\in V$.
\end{notation}

\begin{definition}
We  consider the exterior algebra
\begin{eqnarray*}
{\rm Ext}^*(V)= \bigwedge\Big(dv \mid v\in V\Big)
\end{eqnarray*}
and call it the {\it  co-differential algebra}  on $V$.
\end{definition}

We  have the following  observations:

\begin{itemize}
\item
   ${\rm Ext}^*(V)$  is a direct sum
\begin{eqnarray*}
{\rm Ext}^*(V)=\bigoplus_{k=0}^\infty {\rm Ext}^k(V).
\end{eqnarray*}
\item
${\rm Ext}^0(V)=\mathbb{R}$ while  for each $k\geq 1$,  the space  ${\rm Ext}^k(V)$  is spanned by
\begin{eqnarray*}
d v_1 \wedge d v_2\wedge \cdots \wedge d v_k,~~~ v_1,v_2,\ldots, v_k\in V
\end{eqnarray*}
modulo the relation
\begin{eqnarray*}
d v_1\wedge \cdots \wedge d v_i\wedge d v_{i+1}\wedge \cdots  \wedge d v_k=-dv_1\wedge \cdots \wedge d v_{i+1}\wedge d v_{i}\wedge \cdots  \wedge  d v_k
\end{eqnarray*}
for any $1\leq i\leq k-1$.

\item
  For any $k\geq 1$  and any  $\omega \in {\rm Ext}^k(V)$,  we have     that  $\omega$ gives a sequence of  linear maps
\begin{eqnarray}\label{eq-2.dd}
\omega_n:~~~ \Lambda_n(V)\longrightarrow \Lambda_{n+k}(V),  ~~~ n\geq 0.
\end{eqnarray}

\end{itemize}

\begin{definition}
Let $k\geq 1$,  $\alpha \in {\rm Ext}_k(V)$   and $\omega\in {\rm Ext}^k(V)$.  We  say that $\alpha$ and $\omega$ are {\it adjoint} to each other      if  for any $n\geq 0$,  any  $\xi\in \Lambda_n(V)$  and any $\eta\in \Lambda_{n+k}(V)$,  the identity
\begin{eqnarray*}
\langle \alpha(\eta),\xi \rangle= \langle \eta,  \omega(\xi)\rangle
\end{eqnarray*}
 is satisfied.
\end{definition}

\begin{proposition}\label{le-2.1}
Let $k\geq 1$  be any positive integer.
Let $\alpha\in {\rm Ext}_k(V)$ be  given by
\begin{eqnarray*}
\alpha=\sum_{v_1,v_2,\ldots, v_k\in V} r_{v_1,v_2,\ldots, v_k} \frac{\partial}{\partial v_1}\wedge  \frac{\partial}{\partial  v_2}\wedge \cdots \wedge\frac{\partial}{\partial  v_k},~~~ r_{v_1,v_2,\ldots, v_k}\in\mathbb{R}.
\end{eqnarray*}
Suppose $\omega\in {\rm Ext}^k(V)$ is adjoint to  $\alpha$.  Then $\omega$  is given by
\begin{eqnarray}\label{eq-adj}
\omega={\rm sgn}(k)\sum_{v_1,v_2,\ldots, v_k\in V} r_{v_1,v_2,\ldots, v_k} d v_1\wedge  d  v_2\wedge \cdots \wedge  d v_k,~~~ r_{v_1,v_2,\ldots, v_k}\in\mathbb{R}
\end{eqnarray}
 where ${\rm sgn}(k)=1$  if $k\equiv 0,1$  modulo  $4$  and ${\rm sgn}(k)=-1$  if $k\equiv 2,3$  modulo $4$.
\end{proposition}

\begin{proof}
Let  $n\geq 0$,     $\xi\in \Lambda_n(V)$,  and  $\eta\in \Lambda_{n+k}(V)$.  Then  we  have
\begin{eqnarray*}
\langle  \alpha(\eta),\xi\rangle &=&\sum_{v_1,v_2,\ldots, v_k\in V} r_{v_1,v_2,\ldots,v_k}\langle \frac{\partial}{\partial v_1}\wedge
\frac{\partial}{\partial v_2}\wedge\cdots \wedge \frac{\partial}{\partial v_k}(\eta),\xi\rangle\\
&=&\sum_{v_1,v_2,\ldots, v_k\in V} r_{v_1,v_2,\ldots,v_k}\langle
\frac{\partial}{\partial v_2}\wedge\cdots \wedge \frac{\partial}{\partial v_k}(\eta),dv_1(\xi)\rangle\\
&=&\sum_{v_1,v_2,\ldots, v_k\in V} r_{v_1,v_2,\ldots,v_k}\langle
 \frac{\partial}{\partial v_3}\wedge\cdots \wedge \frac{\partial}{\partial v_k}(\eta), dv_2\wedge dv_1(\xi)\rangle\\
~~\\
&=&\cdots \\
~~\\
&=&\sum_{v_1,v_2,\ldots, v_k\in V} r_{v_1,v_2,\ldots,v_k}\langle
  \eta,  dv_k\wedge dv_{k-1} \wedge \cdots \wedge dv_1(\xi)
\rangle\\
&=&\sum_{v_1,v_2,\ldots, v_k\in V} r_{v_1,v_2,\ldots,v_k}{\rm sgn}(k) \langle
  \eta,   dv_1\wedge dv_{2} \wedge \cdots \wedge dv_k(\xi)
\rangle.
\end{eqnarray*}
The last equality follows from the fact that the permutation  $(k,k-1,\ldots, 1)$   of $(1,2,\ldots,  k)$   has the  signature
\begin{eqnarray*}
{\rm sgn}{{ 1,2,\ldots,k }\choose{ k,k-1,\ldots, 1 }}=(-1)^{(k-1)+(k-2)+\cdots +1}=(-1)^{\frac{k(k-1)}{2}}
\end{eqnarray*}
  for any $k\geq 2$ and the permutation  $(k,k-1,\ldots, 1)$   of $(1,2,\ldots,  k)$   has the signature $1$ for $k=1$.  In other words,  the permutation  $(k,k-1,\ldots, 1)$   of $(1,2,\ldots,  k)$  has   the  signature       $1$ for
   $k\equiv 0,1$  (mod $4$)  and has the  signature  $-1$  for $k\equiv 2,3$  (mod $4$).
Therefore,  we have that the   $\omega$  given by (\ref{eq-adj})  is adjoint to $\alpha$.  The proposition  follows.
\end{proof}

\smallskip

\subsection{Some Chain Complexes   and Co-Chain Complexes on Path  Spaces}

\begin{proposition}\label{le-3.8888x}
Let $t\geq 0$  be a non-negative integer.
Let $\alpha\in {\rm Ext}_{2t+1}(V)$ and $\omega\in {\rm Ext}^{2t+1}(V)$.  Then for any $0\leq q\leq 2t$,    we  have a chain complex
\begin{eqnarray*}
\xymatrix{
\cdots\ar[r]^-{\alpha}
&\Lambda_{n(2t+1)+q}(V)\ar[r]^-{\alpha}
&\Lambda_{(n-1)(2t+1)+q}(V)\ar[r]^-{\alpha}
&\\
\cdots\ar[r]^-{\alpha}
&\Lambda_{(2t+1)+q}(V)\ar[r]^-\alpha
&\Lambda_{q}(V)\ar[r]^-\alpha
&0
}
\end{eqnarray*}
and  a co-chain complex
\begin{eqnarray*}
\xymatrix{
\cdots
&\Lambda_{n(2t+1)+q}(V)\ar[l]_-{\omega}
&\Lambda_{(n-1)(2t+1)+q}(V)\ar[l]_-{\omega}
&  \ar[l]_-{\omega}\\
\cdots&\Lambda_{(2t+1)+q}(V)\ar[l]_-{\omega}
&\Lambda_{q}(V)\ar[l]_-\omega
&\ar[l]_-\omega  0.
}
\end{eqnarray*}
\end{proposition}

\begin{proof}
Let $t\geq 0$.
Let $\alpha\in  {\rm Ext}_{2t+1}(V)$  and $\omega\in {\rm Ext}^{2t+1}(V)$.
Let $0\leq q\leq 2t$.
Note that for each $n\geq 0$,
the maps
\begin{eqnarray*}
\alpha:~~~\Lambda_{n(2t+1)+q}(V)\longrightarrow \Lambda_{(n-1)(2t+1)+q}(V)
\end{eqnarray*}
and
\begin{eqnarray*}
\omega:~~~\Lambda_{n(2t+1)+q}(V)\longrightarrow \Lambda_{(n+1)(2t+1)+q}(V)
\end{eqnarray*}
are well-defined.
By the anti-symmetric property of exterior algebras,
we  have
\begin{eqnarray*}
\alpha\wedge \alpha=(-1)^{(2t+1)^2} \alpha\wedge\alpha,~~~~~~  \omega\wedge\omega=(-1)^{(2t+1)^2}\omega\wedge\omega.
\end{eqnarray*}
Since $(2t+1)^2$  is odd,  we have
\begin{eqnarray*}
\alpha\circ\alpha=\alpha\wedge \alpha=0,~~~~~~ \omega\circ\omega=\omega\wedge\omega=0.
\end{eqnarray*}
  Thus for any $0\leq q\leq 2t$,  we have the chain complex  as well as  the co-chain complex as  given   in the proposition.
\end{proof}

\begin{notation}
Let   $0\leq q\leq 2t$.  Let $\alpha\in {\rm Ext}_{2t+1}(V)$ and $\omega\in {\rm Ext}^{2t+1}(V)$.    We adopt  the following  notations:
\begin{enumerate}[(i).]
\item
denote the chain complex  in Proposition~\ref{le-3.8888x}  as
\begin{eqnarray*}
\Lambda_*(V,\alpha,q)=\{\Lambda_{n(2t+1)+q}(V),\alpha\}_{n\geq 0};
\end{eqnarray*}
\item
 denote the co-chain complex  in Proposition~\ref{le-3.8888x}  as
\begin{eqnarray*}
\Lambda^*(V,\omega,q)=\{\Lambda_{n(2t+1)+q}(V),\omega\}_{n\geq 0}.
\end{eqnarray*}
\end{enumerate}
\end{notation}

\begin{notation}
For any integer $m$,  there is a unique integer $\lambda$ (not necessarily non-negative)  and a unique  integer  $0\leq q\leq 2t$  such that $m=\lambda(2t+1)+q$.   We  adopt  the following notations:
\begin{enumerate}[(i).]
\item
denote the chain complex
\begin{eqnarray*}
\xymatrix{
\cdots\ar[r]^-{\alpha}
&\Lambda_{(n+\lambda)(2t+1)+q}(V)\ar[r]^-{\alpha}
&\Lambda_{(n-1+\lambda)(2t+1)+q}(V)\ar[r]^-{\alpha}
&\\
\cdots\ar[r]^-{\alpha}
&\Lambda_{(1+\lambda)(2t+1)+q}(V)\ar[r]^-\alpha
&\Lambda_{\lambda(2t+1)+q}(V)\ar[r]^-\alpha
&0
}
\end{eqnarray*}
as
\begin{eqnarray*}
\Lambda_*(V,\alpha,m)=\{\Lambda_{(n+\lambda)(2t+1)+q} (V),\alpha\}_{n\geq 0};
\end{eqnarray*}
\item
denote the  co-chain complex
\begin{eqnarray*}
\xymatrix{
\cdots
&\Lambda_{(n+\lambda)(2t+1)+q}(V)\ar[l]_-{\omega}
&\Lambda_{(n-1+\lambda)(2t+1)+q}(V)\ar[l]_-{\omega}
&  \ar[l]_-{\omega}\\
\cdots&\Lambda_{(1+\lambda)(2t+1)+q}(V)\ar[l]_-{\omega}
&\Lambda_{\lambda(2t+1)q}(V)\ar[l]_-\omega
&\ar[l]_-\omega  0
}
\end{eqnarray*}
as
\begin{eqnarray*}
\Lambda^*(V,\omega,m)=\{\Lambda_{(n+\lambda)(2t+1)+q} (V),\omega\}_{n\geq 0}.
\end{eqnarray*}
\end{enumerate}
Here in both (i)  and (ii),  we  use the notation $\Lambda_k(V)=0$  for $k<0$.
\end{notation}

\begin{proposition}\label{pr-3.26}
Let $t,s\geq 0$  be   non-negative integers.  Let $m\in\mathbb{Z}$.
Let $\alpha\in {\rm Ext}_{2t+1}(V)$ and $\omega\in {\rm Ext}^{2t+1}(V)$.    Let  $\beta\in {\rm Ext}_{2s}(V)$  and $\mu\in {\rm Ext}^{2s}(V)$.  Then $\beta$  gives  a  chain map
\begin{eqnarray*}
\beta:~~~  \Lambda_*(V,\alpha,m)\longrightarrow \Lambda_*(V,\alpha,m-2s)
\end{eqnarray*}
and $\mu$  gives a   co-chain map
 \begin{eqnarray*}
\mu:~~~  \Lambda^*(V,\omega,m)\longrightarrow \Lambda^*(V,\omega,m+2s).
\end{eqnarray*}
\end{proposition}

\begin{proof}
Note that  as  linear maps,
\begin{eqnarray*}
\beta:~~~\Lambda_{(n+\lambda)(2t+1)+q} (V)\longrightarrow \Lambda_{(n+\lambda)(2t+1)+q-2s} (V)
\end{eqnarray*}
and
\begin{eqnarray*}
\mu:~~~\Lambda_{(n+\lambda)(2t+1)+q} (V)\longrightarrow \Lambda_{(n+\lambda)(2t+1)+q+2s} (V)
\end{eqnarray*}
are well-defined.
By the anti-symmetric property of exterior algebras,  we  have  (cf.  \cite[p,  53,  Anticommutative Law]{chern})
\begin{eqnarray*}
\alpha\wedge \beta=(-1)^{2s(2t+1)} \beta\wedge\alpha=\beta\wedge\alpha.
\end{eqnarray*}
That is,
\begin{eqnarray*}
\alpha\circ\beta=\beta\circ\alpha.
\end{eqnarray*}
Thus $\beta$  is a chain map  from  $\Lambda_*(V,\alpha,m)$  to  $\Lambda_*(V,\alpha,m-2s)$.
 Moreover,  we also have    (cf.  \cite[p,  53,  Anticommutative Law]{chern})
\begin{eqnarray*}
&\omega\wedge\mu=(-1)^{2s(2t+1)}\mu\wedge\omega=\mu\wedge\omega.
\end{eqnarray*}
That is,
\begin{eqnarray*}
\omega\circ\mu=\mu\circ\omega.
\end{eqnarray*}
Thus $\mu$  is a  co-chain map  from  $\Lambda^*(V,\omega,m)$  to  $\Lambda^*(V,\omega,m+2s)$.    The proposition follows.
\end{proof}

\smallskip

\section{Constrained  Homology for Simplicial Complexes and  Constrained Cohomology for Independent Hypergraphs}\label{s4}

In  this  section,   we  define  the constrained homology groups for simplicial complexes and the constrained cohomology groups for independent hypergraphs.  We  prove that any element $\beta\in {\rm Ext}_{2s}(V)$,  where $s\geq 0$,   induces  homomorphisms  between the constrained homology groups  for the  simplicial complexes on $V$.  We also  prove that any element $\mu\in {\rm Ext}^{2s}(V)$,  where $s\geq 0$,   induces  homomorphisms  between the constrained cohomology  groups  for the  independent hypergraphs on $V$.

\smallskip

\subsection{Some  Auxiliaries}

Throughout this section,  we 
let $V$  be a finite set.   Let $\Delta[V]$  be the complete hypergraph  on $V$.  For  each   integer  $n\geq 0$,  let
\begin{eqnarray*}
C_n(\Delta[V];\mathbb{R})={\rm Span}_{\mathbb{R}}\{\sigma^{(n)}\in \Delta[V]\}
\end{eqnarray*}
be  the vector space  consisting of all the linear combinations of the $n$-hyperedges on $V$.   Consider   the  direct  sum
\begin{eqnarray}\label{eq-aq}
C_*(\Delta[V];\mathbb{R})=\bigoplus_{n\geq 0} C_n(\Delta[V];\mathbb{R}).
\end{eqnarray}
Note that since $V$  is  assumed to be  a finite set,  the direct sum in the right-hand side  of  (\ref{eq-aq})  is  a  finite sum.

\begin{lemma}\label{le-4.1cc}
Let  $t\geq 0$  be a  non-negative integer.   Let  $m\in \mathbb{Z}$.    Suppose $m=\lambda(2t+1)+q$  where $\lambda\in \mathbb{Z}$  and  the integer $0\leq q\leq 2t$.  Then
for any $\alpha\in {\rm Ext}_{2t+1}(V)$,
the graded vector space
\begin{eqnarray}\label{eq-dt-1}
C_{(n+\lambda)(2t+1)+q}(\Delta[V];\mathbb{R}),~~~~~~  n\geq 0
\end{eqnarray}
equipped with the boundary map  $\alpha$  gives a sub-chain  complex of $\Lambda_*(V,\alpha, m)$,  which  will be denoted as    $C_*(\Delta[V],\alpha,m)$.
\end{lemma}

\begin{proof}
For each $n\geq 0$,  the vector space  $C_{(n+\lambda)(2t+1)+q}(\Delta[V];\mathbb{R})$   is a  subspace of  the vector space  $\Lambda_{(n+\lambda)(2t+1)+q}(V)$.    Hence  in  order to prove  that (\ref{eq-dt-1})   equipped with $\alpha$   is  a  sub-chain  complex  of     $\Lambda_*(V,\alpha, m)$,   it  suffices to prove  that the map
\begin{eqnarray}\label{eq-wd}
\alpha:~~~ C_{(n+\lambda)(2t+1)+q}(\Delta[V];\mathbb{R})\longrightarrow C_{(n-1+\lambda)(2t+1)+q}(\Delta[V];\mathbb{R})
\end{eqnarray}
is  well-defined for each  $n\geq 0$.
This follows from the observation that for any        $[(n+\lambda)(2t+1)+q]$-simplex
\begin{eqnarray*}
  v_0v_1\ldots v_{(n+\lambda)(2t+1)+q} \in C_{(n+\lambda)(2t+1)+q}(\Delta[V];R)
\end{eqnarray*}
    and any
   \begin{eqnarray*}
\alpha=\frac{\partial}{\partial u_1} \wedge  \frac{\partial}{\partial  u_2}\wedge \cdots \wedge \frac{\partial}{\partial u_{2t+1}}
\end{eqnarray*}
where    $u_1, u_2,\ldots, u_{2t+1}\in V$  and  $u_1\prec u_2\prec\cdots \prec  u_{2t+1}$,     we  have
\begin{eqnarray*}
\alpha  (v_0v_1\ldots v_{(n+\lambda)(2t+1)+q} ) \in C_{(n-1+\lambda)(2t+1)+q}(\Delta[V];R).
\end{eqnarray*}
By  a  calculation of  linear combinations,   it follows that   the  map   (\ref{eq-wd})   is  well-defined.   Therefore,    the  graded   vector space    (\ref{eq-dt-1})   equipped with $\alpha$   is  a  sub-chain  complex  of     $\Lambda_*(V,\alpha, m)$.
\end{proof}

\begin{definition}
For  any  $n\geq 0$  and any elementary  $n$-path $v_0 v_1\ldots v_n$  on $V$,  we call   $v_0 v_1\ldots v_n$  a  {\it  non-simplicial} elementary  $n$-path  if there exist integers  $0\leq i<j\leq n$  such that  either  $v_j\prec v_i$  or  $v_j=v_i$.
\end{definition}

\begin{definition}
Let $\mathcal{O}_n(V)$  be   the vector space spanned by all the  non-simplicial   elementary  $n$-paths   on $V$.  Then     $\mathcal{O}_n(V)$  consists of all the   linear combinations  of the  non-simplicial   elementary  $n$-paths   on $V$.
 We  call  an  element   in  $\mathcal{O}_n(V)$  a  {\it  non-simplicial    $n$-path}  on $V$.
\end{definition}

\begin{lemma}\label{le-4.cd}
Let  $t\geq 0$  be  an  integer.   Let  $m\in \mathbb{Z}$.    Suppose $m=\lambda(2t+1)+q$  where $\lambda\in \mathbb{Z}$  and the integer $0\leq q\leq 2t$.  Then
for any $\omega\in {\rm Ext}^{2t+1}(V)$,
the graded vector space
\begin{eqnarray}\label{eq-4.97}
\mathcal{O}_{(n+\lambda)(2t+1)+q}(V),~~~~~~  n\geq 0
\end{eqnarray}
equipped with  the  co-boundary map  $\omega$  gives a sub-co-chain  complex of $\Lambda^*(V,\omega, m)$,  which will be  denoted as    $\mathcal{O}^*(V,\omega,m)$.
\end{lemma}

\begin{proof}
It suffices to verify that  the  map
\begin{eqnarray}\label{eq-4.98}
\omega:~~~ \mathcal{O}_{(n+\lambda)(2t+1)+q}(V)\longrightarrow \mathcal{O}_{(n+1+\lambda)(2t+1)+q}(V)
\end{eqnarray}
is well-defined for each $n\geq 0$.   This follows from the observation that  after adding some  vertices to any non-simplicial   elementary path, we still get a non-simplicial  elementary path.
 Hence for any  non-simplicial  elementary    $[(n+\lambda)(2t+1)+q]$-path
\begin{eqnarray*}
  v_0v_1\ldots v_{(n+\lambda)(2t+1)+q} \in \mathcal{O}_{(n+\lambda)(2t+1)+q}(V)
\end{eqnarray*}
    and any
   \begin{eqnarray*}
\omega=d u_1 \wedge  du_2\wedge \cdots \wedge d u_{2t+1}
\end{eqnarray*}
where    $u_1, u_2,\ldots, u_{2t+1}\in V$  and  $u_1\prec u_2\prec\cdots \prec  u_{2t+1}$,     we  have
\begin{eqnarray*}
\omega  (v_0v_1\ldots v_{(n+\lambda)(2t+1)+q} ) \in \mathcal{O}_{(n+1+\lambda)(2t+1)+q}(V).
\end{eqnarray*}
By  a  calculation of  linear combinations,  it follows that  the  map  (\ref{eq-4.98})  is well-defined.   Therefore,  the graded vector space  (\ref{eq-4.97})  equipped with $\omega$  is a sub-co-chain complex of   $\Lambda^*(V,\omega, m)$.
\end{proof}

\begin{lemma}\label{le-4.91}
Let  $t\geq 0$  be an integer.   Let  $m\in \mathbb{Z}$.    Suppose $m=\lambda(2t+1)+q$  where $\lambda\in \mathbb{Z}$  and  the  integer  $0\leq q\leq 2t$.  Then
for any $\omega\in {\rm Ext}^{2t+1}(V)$,
the graded vector space
\begin{eqnarray}\label{eq-4.95}
C_{(n+\lambda)(2t+1)+q}(\Delta[V];\mathbb{R}),~~~~~~  n\geq 0
\end{eqnarray}
equipped with  the  co-boundary  map  $\omega$  gives a quotient co-chain  complex   $\Lambda^*(V,\omega, m)/\mathcal{O}^*(V,\omega,m)$  which will be denoted as $C^*(\Delta[V],\omega,m)$.
\end{lemma}

\begin{proof}
Note that the  canonical  inclusion of  the   sub-co-chain complex $\mathcal{O}^*(V,\omega,m)$  into   the co-chain  complex  $\Lambda^*(V,\omega, m)$  gives a quotient co-chain complex $\Lambda^*(V,\omega, m)/\mathcal{O}^*(V,\omega,m)$.  On the other hand,  for each $n\geq 0$,   the  quotient  vector space
\begin{eqnarray*}
\Lambda_{(n+\lambda)(2t+1)+q}(V)/\mathcal{O}_{(n+\lambda)(2t+1)+q}(V)
\end{eqnarray*}
   is   canonically  isomorphic to  the vector space  $C_{(n+\lambda)(2t+1)+q}(\Delta[V];\mathbb{R})$.  Therefore,  the  quotient co-chain complex $\Lambda^*(V,\omega, m)/\mathcal{O}^*(V,\omega,m)$  is given by  the graded vector space (\ref{eq-4.95})  equipped with the  co-boundary  map  $\omega$.   The lemma follows.
\end{proof}

With the help of Proposition~\ref{pr-3.26},  the next proposition follows.

\begin{proposition}\label{pr-4.26}
Let $t,s\geq 0$  be   non-negative integers.  Let $m\in\mathbb{Z}$.
Let $\alpha\in {\rm Ext}_{2t+1}(V)$ and $\omega\in {\rm Ext}^{2t+1}(V)$.    Let  $\beta\in {\rm Ext}_{2s}(V)$  and $\mu\in {\rm Ext}^{2s}(V)$.  Then $\beta$ is a  chain map
\begin{eqnarray}\label{eq-4.81}
\beta:~~~  C_*(\Delta[V],\alpha,m)\longrightarrow  C_*(\Delta[V],\alpha,m-2s)
\end{eqnarray}
and $\mu$ is a   co-chain map
 \begin{eqnarray}\label{eq-4.82}
\mu:~~~  C^*(\Delta[V],\omega,m)\longrightarrow C^*(\Delta[V],\omega,m+2s).
\end{eqnarray}
\end{proposition}

 \begin{proof}
By a similar argument  in the proof of  Lemma~\ref{le-4.1cc},  it can be verified  that   as  a  linear map,  $\beta$   in  (\ref{eq-4.81})   is well-defined.  Thus  by Proposition~\ref{pr-3.26}  and Lemma~\ref{le-4.1cc},   it follows that  $\beta$   in  (\ref{eq-4.81})  is a  chain map.   On the other hand,  by a similar argument  in the proof of  Lemma~\ref{le-4.91},  it can be verified  that   as  a  linear map,  $\mu$   in  (\ref{eq-4.82})   is well-defined.   Thus  by Proposition~\ref{pr-3.26}  and Lemma~\ref{le-4.91}, it follows that  $\omega$   in  (\ref{eq-4.82})  is a  co-chain map.
 \end{proof}

\smallskip

\subsection{Constrained Homology for Simplicial Complexes}\label{ss4.2}

Let $\mathcal{K}$  be a simplicial  complex  with its vertices in $V$.

\begin{notation}
  For each  non-negative  integer  $n\geq 0$,  let  $C_n(\mathcal{K};\mathbb{R})$  be the (real)  vector space consisting of all the linear  combinations of the $n$-simplices in $\mathcal{K}$.
\end{notation}

\begin{theorem}\label{th-4.1}
Let  $t,s\geq 0$ be  non-negative  integers.   Let  $m\in \mathbb{Z}$.    Suppose $m=\lambda(2t+1)+q$  where $\lambda\in \mathbb{Z}$  and  $0\leq q\leq 2t$.  Then
\begin{enumerate}[(i).]
\item
for any $\alpha\in {\rm Ext}_{2t+1}(V)$,
the graded vector space
\begin{eqnarray}\label{eq-dt-9a}
C_{(n+\lambda)(2t+1)+q}(\mathcal{K};\mathbb{R}),~~~~~~  n\geq 0
\end{eqnarray}
equipped with the  chain  map  $\alpha$  gives a sub-chain  complex   of $C_*(\Delta[V],\alpha, m)$,  which will be  denoted  as $C_*(\mathcal{K},\alpha, m)$;
\item
  for any $\beta\in {\rm Ext}_{2s}(V)$,  there is  an  induced   chain map
\begin{eqnarray}\label{eq-4.89a}
\beta:~~~  C_*(\mathcal{K},\alpha,m)\longrightarrow  C_*(\mathcal{K},\alpha,m-2s).
\end{eqnarray}
\end{enumerate}
\end{theorem}

\begin{proof}
We  prove (i)  and (ii)  subsequently.

(i).
For each $n\geq 0$,  the vector space  $C_{(n+\lambda)(2t+1)+q}(\mathcal{K};\mathbb{R})$  is  a  subspace  of  the vector space  $C_{(n+\lambda)(2t+1)+q}(\Delta[V];\mathbb{R})$.
  Hence  in  order to prove  that   the  graded vector  space  (\ref{eq-dt-9a})   equipped with  the  chain  map  $\alpha$   is  a  sub-chain  complex  of     $C_*(\Delta[V],\alpha, m)$,   it  suffices to prove  that  the  map
\begin{eqnarray}\label{eq-wd-aa}
\alpha:~~~ C_{(n+\lambda)(2t+1)+q}(\mathcal{K};\mathbb{R})\longrightarrow C_{(n-1+\lambda)(2t+1)+q}(\mathcal{K};\mathbb{R})
\end{eqnarray}
is  well-defined for each  $n\geq 0$.  This follows from the observation that for any        $[(n+\lambda)(2t+1)+q]$-simplex
\begin{eqnarray*}
  v_0v_1\ldots v_{(n+\lambda)(2t+1)+q} \in C_{(n+\lambda)(2t+1)+q}(\mathcal{K};R)
\end{eqnarray*}
    and any
   \begin{eqnarray*}
\alpha=\frac{\partial}{\partial u_1} \wedge  \frac{\partial}{\partial  u_2}\wedge \cdots \wedge \frac{\partial}{\partial u_{2t+1}}
\end{eqnarray*}
where    $u_1, u_2,\ldots, u_{2t+1}\in V$  and  $u_1\prec u_2\prec\cdots \prec  u_{2t+1}$,     we  have
\begin{eqnarray*}
\alpha  (v_0v_1\ldots v_{(n+\lambda)(2t+1)+q} ) \in C_{(n-1+\lambda)(2t+1)+q}(\mathcal{K};R).
\end{eqnarray*}
By  a  calculation of  linear combinations,   it follows that   the  map  (\ref{eq-wd-aa})   is  well-defined.   Therefore,    the  graded   vector space    (\ref{eq-dt-9a})   equipped with  the  chain  map  $\alpha$   is  a  sub-chain  complex  of    $C_*(\Delta[V],\alpha, m)$.

(ii).  Similar with the   verification that    the  map  $\alpha$  in (\ref{eq-wd-aa})   is well-defined for each $n\geq 0$,  we can prove that  the map
\begin{eqnarray*}
\beta:~~~ C_{(n+\lambda)(2t+1)+q}(\mathcal{K};\mathbb{R})\longrightarrow C_{(n+\lambda)(2t+1)+q-2s}(\mathcal{K};\mathbb{R})
\end{eqnarray*}
is  well-defined for each  $n\geq 0$.  Therefore,  with the help of (\ref{eq-4.81})  in Proposition~\ref{pr-4.26},   we  have that  $\beta$  gives a chain map   in  (\ref{eq-4.89a}).
\end{proof}

\begin{definition}\label{def-4.11}
Let  $t\geq 0$  be a  non-negative integer.    Let $\alpha\in  {\rm Ext}_{2t+1}(V)$.  Let  $m\in \mathbb{Z}$.   Suppose $m=\lambda(2t+1)+q$  where $\lambda\in \mathbb{Z}$  and  $0\leq q\leq 2t$.   Let $\mathcal{K}$  be a simplicial complex with its vertices in $V$.  For each $n\geq 0$,  we  define the  $n$-th  {\it  constrained homology group} $H_n(\mathcal{K},\alpha,m)$ of $\mathcal{K}$  with respect to $\alpha$ and $m$   to be the  $n$-th   homology group
\begin{eqnarray*}
H_n(\mathcal{K},\alpha,m):&=&H_n(C_*(\mathcal{K},\alpha, m))  \\
&=& \frac{{\rm Ker}\Big(\alpha:C_{(n+\lambda)(2t+1)+q}(\mathcal{K};\mathbb{R})\longrightarrow C_{(n-1+\lambda)(2t+1)+q}(\mathcal{K};\mathbb{R})\Big)}{ {\rm Im}\Big(\alpha:C_{(n+1+\lambda)(2t+1)+q}(\mathcal{K};\mathbb{R})\longrightarrow C_{(n+\lambda)(2t+1)+q}(\mathcal{K};\mathbb{R})\Big) }
\end{eqnarray*}
of the chain complex $C_*(\mathcal{K},\alpha, m)$.
\end{definition}

The next theorem follows  from  Theorem~\ref{th-4.1}  and  Definition~\ref{def-4.11} immediately.

\begin{theorem}[Main Result I]
\label{th-4.1h}
Let   $t,s\geq 0$  be  non-negative  integers.   Let  $m\in \mathbb{Z}$.    Suppose $m=\lambda(2t+1)+q$  where $\lambda\in \mathbb{Z}$  and  $0\leq q\leq 2t$.  Then for any $\alpha\in {\rm Ext}_{2t+1}(V)$  and $\beta\in {\rm Ext}_{2s}(V)$,  there is an  induced homomorphism
\begin{eqnarray}\label{eq-722}
\beta_*: ~~~H_n(\mathcal{K},\alpha,m)\longrightarrow  H_n(\mathcal{K},\alpha,m-2s), ~~~~~~  n\geq 0
\end{eqnarray}
of  the  constrained  homology groups.
\end{theorem}

\begin{proof}
Apply the homology functor to the chain complex in Theorem~\ref{th-4.1}~(i) and   the chain map  in Theorem~\ref{th-4.1}~(ii).  We  obtain the homomorphism  $\beta_*$   of the   constrained homology groups  in (\ref{eq-722}).
\end{proof}

\smallskip

\subsection{Constrained  Cohomology for Independent Hypergraphs}\label{ss4.3}

Let $\mathcal{L}$  be a  independent hypergraph  with its vertices in $V$.

\begin{notation}
  For each  non-negative  integer $n\geq 0$,  let  $C_n(\mathcal{L};\mathbb{R})$  be the (real)  vector space consisting of all the linear  combinations of the $n$-hyperedges in $\mathcal{L}$.
\end{notation}

\begin{theorem}\label{th-4.2}
Let  $t,s\geq 0$  be  non-negative  integers.   Let  $m\in \mathbb{Z}$.    Suppose $m=\lambda(2t+1)+q$  where $\lambda\in \mathbb{Z}$  and  $0\leq q\leq 2t$.  Then
\begin{enumerate}[(i).]
\item
for any $\omega\in {\rm Ext}^{2t+1}(V)$,
the graded vector space
\begin{eqnarray}\label{eq-dt-999}
C_{(n+\lambda)(2t+1)+q}(\mathcal{L};\mathbb{R}),~~~~~~  n\geq 0
\end{eqnarray}
equipped with  the  co-boundary  map  $\omega$  gives a sub-co-chain  complex   of $C^*(\Delta[V],\omega, m)$,  which will be denoted  as $C^*(\mathcal{L},\omega, m)$;
\item
 for any $\mu\in {\rm Ext}^{2s}(V)$,  there is  an  induced  co-chain map
\begin{eqnarray}\label{eq-4.8999}
\mu:~~~  C^*(\mathcal{L},\omega,m)\longrightarrow  C^*(\mathcal{L},\omega,m+2s).
\end{eqnarray}
\end{enumerate}
\end{theorem}

\begin{proof}
We  prove (i)  and (ii)  subsequently.

(i).  For each $n\geq 0$,  the vector space  $C_{(n+\lambda)(2t+1)+q}(\mathcal{L};\mathbb{R})$  is  a  subspace  of  the vector space  $C_{(n+\lambda)(2t+1)+q}(\Delta[V];\mathbb{R})$.
  Hence  in  order to prove  that   the  graded vector  space  (\ref{eq-dt-999})   equipped with  the  co-boundary  map  $\omega$   is  a  sub-co-chain  complex  of     $C^*(\Delta[V],\omega, m)$,   it  suffices to prove  that  the  map
\begin{eqnarray}\label{eq-wd-aazz}
\omega:~~~ C_{(n+\lambda)(2t+1)+q}(\mathcal{L};\mathbb{R})\longrightarrow C_{(n+1+\lambda)(2t+1)+q}(\mathcal{L};\mathbb{R})
\end{eqnarray}
is  well-defined for each  $n\geq 0$.  This follows from the observation that for any        $[(n+\lambda)(2t+1)+q]$-hyperedge 
\begin{eqnarray*}
  v_0v_1\ldots v_{(n+\lambda)(2t+1)+q} \in C_{(n+\lambda)(2t+1)+q}(\mathcal{L};R)
\end{eqnarray*}
    and any
    \begin{eqnarray*}
\omega=d u_1 \wedge  d u_2 \wedge \cdots \wedge  d  u_{2t+1}
\end{eqnarray*}
where    $u_1, u_2,\ldots, u_{2t+1}\in V$  and  $u_1\prec u_2\prec\cdots \prec u_{2t+1}$,     we  have
\begin{eqnarray*}
\omega  (v_0v_1\ldots v_{(n+\lambda)(2t+1)+q} ) \in C_{(n+1+\lambda)(2t+1)+q}(\mathcal{L};R).
\end{eqnarray*}
By  a  calculation of  linear combinations,   it follows that  the map  (\ref{eq-wd-aazz})   is  well-defined.   Therefore,    the  graded   vector space    (\ref{eq-dt-999})   equipped with the  co-boundary map  $\omega$   is  a  sub-co-chain  complex  of    $C^*(\Delta[V],\omega, m)$.

(ii).  Similar with the   verification that    the  map  $\omega$  in (\ref{eq-wd-aazz})   is well-defined for each $n\geq 0$,  we can prove that  the map
\begin{eqnarray*}
\mu:~~~ C_{(n+\lambda)(2t+1)+q}(\mathcal{L};\mathbb{R})\longrightarrow C_{(n+\lambda)(2t+1)+q+2s}(\mathcal{L};\mathbb{R})
\end{eqnarray*}
is  well-defined for each  $n\geq 0$.  Therefore,  with the help of (\ref{eq-4.82})  in Proposition~\ref{pr-4.26},   we  have that  $\mu$  gives a  co-chain map   in  (\ref{eq-4.8999}).
\end{proof}

\begin{definition}\label{def-4.22}
Let  $t\geq 0$.    Let $\omega\in  {\rm Ext}^{2t+1}(V)$.  Let  $m\in \mathbb{Z}$.   Suppose $m=\lambda(2t+1)+q$  where $\lambda\in \mathbb{Z}$  and  $0\leq q\leq 2t$.   Let $\mathcal{L}$  be a   independent hypergraph with its vertices in $V$.  For each $n\geq 0$,  we  define the  $n$-th  {\it  constrained  cohomology group}  $H^n(\mathcal{L},\omega, m)$  of $\mathcal{L}$  with respect to $\omega$ and $m$   to be the  cohomology group
\begin{eqnarray*}
H^n(\mathcal{L},\omega, m):&=&H^n(C^*(\mathcal{L},\omega, m))\\
&  =& \frac{{\rm Ker}\Big(\omega:  C_{(n+\lambda)(2t+1)+q}(\mathcal{L};\mathbb{R})\longrightarrow C_{(n+1+\lambda)(2t+1)+q}(\mathcal{L};\mathbb{R})\Big)}{ {\rm Im}\Big(\omega: C_{(n-1+\lambda)(2t+1)+q}(\mathcal{L};\mathbb{R})\longrightarrow C_{(n+\lambda)(2t+1)+q}(\mathcal{L};\mathbb{R})\Big) }
\end{eqnarray*}
of the  co-chain complex $C^*(\mathcal{L},\omega, m)$.
\end{definition}

The next theorem follows  from  Theorem~\ref{th-4.2}  and  Definition~\ref{def-4.22} immediately.

\begin{theorem}[Main Result  II]
\label{th-4.2h}
Let  $t,s\geq 0$  be  non-negative  integers.   Let  $m\in \mathbb{Z}$  be  non-negative  integers.    Suppose $m=\lambda(2t+1)+q$  where $\lambda\in \mathbb{Z}$  and  $0\leq q\leq 2t$.  Then for any $\omega\in {\rm Ext}^{2t+1}(V)$  and $\mu\in {\rm Ext}^{2s}(V)$,  there is an  induced homomorphism
\begin{eqnarray}\label{eq-722a}
\mu_*: ~~~H^n(\mathcal{L},\omega,m)\longrightarrow  H^n(\mathcal{L},\omega,m+2s), ~~~~~~  n\geq 0
\end{eqnarray}
of  the  constrained  cohomology groups.
\end{theorem}

\begin{proof}
Apply the  cohomology functor to the  co-chain complex in Theorem~\ref{th-4.2}~(i) and   the  co-chain map  in Theorem~\ref{th-4.2}~(ii).  We  obtain the homomorphism  $\mu_*$  of the  constrained  cohomology groups  in (\ref{eq-722a}).
\end{proof}

\smallskip

\section{Examples}\label{s5}

We  give some examples  for Theorem~\ref{th-4.1},  Theorem~\ref{th-4.1h},  Theorem~\ref{th-4.2}  and Theorem~\ref{th-4.2h}.

\begin{example}\label{ex-4.3.1}
Let $V$  be  any finite set.
Then  we  have the followings.
\begin{enumerate}[(i).]
\item
Any   element $\alpha\in {\rm Ext}_1(V)$  can be expressed  as
\begin{eqnarray}\label{eq-5.1}
\alpha=\sum_{v\in V}f(v)\frac{\partial}{\partial v}
\end{eqnarray}
for some function $f: V\longrightarrow\mathbb{R}$.  Let $\mathcal{K}$ be a simplicial complex with its vertices in $V$.  Then for any $n\geq 0$ and any $n$-simplex  $v_0v_1\ldots v_n$  in $\mathcal{K}$,  we  have
\begin{eqnarray*}
\alpha (v_0v_1\ldots v_n)&=&\sum_{v\in V}f(v)\frac{\partial}{\partial v} (v_0v_1\ldots v_n)\\
&=&\sum_{v\in V} f(v)\sum_{i=0}^n (-1)^i\delta (v,v_i) v_0\ldots \widehat{v_i} \ldots v_n\\
&=&\sum_{i=0}^n(-1)^i \Big(\sum_{v\in V} \delta(v,v_i) f(v)\Big)v_0\ldots \widehat{v_i} \ldots v_n\\
&=&\sum_{i=0}^n (-1)^i f(v_i) v_0\ldots \widehat{v_i} \ldots v_n.
\end{eqnarray*}
In \cite{chengyuan1,chengyuan2,chengyuan3},  the  $\alpha$ given in (\ref{eq-5.1})  is called  the {\it $f$-weighted boundary operator}  on $\mathcal{K}$ and  the $(\alpha,0)$-homology of $\mathcal{K}$   is denoted  as  the  weighted homology $H_*(\mathcal{K},f)$ of the weighted simplicial  complex $(\mathcal{K},f)$.  Particularly,  if $f$ takes the  constant value $1$ for all $v\in V$,  then $\alpha$  is the usual boundary operator $\partial_*$  given in (\ref{eq-usual-boundary})  and $H_*(\mathcal{K},f)$  is the usual homology  $H_*(\mathcal{K})$   (cf. \cite[Chapter~1]{eat}  and  \cite[Section~2.1]{hatcher})  of $\mathcal{K}$.
\item
Any   element $\omega\in {\rm Ext}^1(V)$  can be expressed  as
\begin{eqnarray}\label{eq-5.1co}
\omega=\sum_{v\in V}f(v) d v
\end{eqnarray}
for some function $f: V\longrightarrow\mathbb{R}$.  Let $\mathcal{L}$ be a  independent hypergraph with its vertices in $V$.  Then for any $n\geq 0$ and any $n$-hyperedge  $v_0v_1\ldots v_n$  in $\mathcal{L}$,  we  have
\begin{eqnarray*}
\omega (v_0v_1\ldots v_n)&=&\sum_{v\in V}f(v) d v  (v_0v_1\ldots v_n)\\
&=&\sum_{v\in V} f(v)\sum_{i=0}^{n+1}   (-1)^i v_0v_1\ldots v_{i-1} v v_i v_{i+1}\ldots v_{n}
\\
&=&\sum_{i=0}^{n+1}(-1)^i \Big(\sum_{v\in V}   f(v)  v_0\ldots \ldots v_{i-1} v v_i \ldots v_n\Big).
\end{eqnarray*}
Similar with (i),  we call the $\omega$ given in (\ref{eq-5.1co})    the {\it $f$-weighted  co-boundary operator}  on $\mathcal{L}$ and  denote  the $(\omega,0)$-cohomology of $\mathcal{L}$     as $H^*(\mathcal{L},f)$.  Particularly,  if $f$ takes the  constant value $1$ for all $v\in V$,  then we denote the $\omega$  as  $d_*$  denote the  $H^*(\mathcal{L},f)$  as  $H^*(\mathcal{L})$.
\end{enumerate}
\end{example}

\begin{example}\label{ex-5.2}
Let $V=\{v_0,v_1,v_2\}$.  Let $f: V\longrightarrow\mathbb{R}$  be a function on $V$.
\begin{enumerate}[(i).]
\item
 Let
\begin{eqnarray*}
\mathcal{K}=\{v_0,v_1,v_2,v_0v_1,v_0v_2,v_1v_2\}
\end{eqnarray*}
  be a simplicial complex  with its vertices in $V$.   Then we   have
  \begin{eqnarray*}
  C_0(\mathcal{K};\mathbb{R})&=& {\rm Span}_{\mathbb{R}}\{v_0,v_1,v_2\}, \\
C_1(\mathcal{K};\mathbb{R})&=& {\rm Span}_{\mathbb{R}}\{v_0v_1, v_0v_2, v_1v_2\},\\
        C_n(\mathcal{K};\mathbb{R})&=&0 {\rm~~~~ for ~all~ } n\geq 2.
  \end{eqnarray*}
  \begin{itemize}
  \item
  Let $t=1$.  Let
  \begin{eqnarray*}
  \alpha&=&\sum_{v\in V}  f(v)\frac{\partial}{\partial v}\\
  &=&f(v_0)\frac{\partial}{\partial v_0} + f(v_1)\frac{\partial}{\partial v_1} +f(v_2)\frac{\partial}{\partial v_2}.
  \end{eqnarray*}
 With the help of  Example~\ref{ex-4.3.1}~(i),   we  have
 \begin{eqnarray*}
 &\alpha (v_0)=\alpha(v_1)=\alpha(v_2)=0,\\
 &\alpha(v_0v_1)=f(v_0) v_1-f(v_1)v_0,\\
  &\alpha(v_0v_2)=f(v_0) v_2-f(v_2)v_0,\\
  &\alpha(v_1v_2)=f(v_1) v_2-f(v_2)v_1,\\
 &\alpha(v_0v_1v_2)= f(v_0)v_1v_2-f(v_1)v_0v_2+f(v_2)v_0v_1.
 \end{eqnarray*}
 Note that
 \begin{eqnarray*}
 \dim {\rm Ker}\Big(\alpha :  C_0(\mathcal{K};\mathbb{R})\longrightarrow 0\Big)= 3
 \end{eqnarray*}
 and
 \begin{eqnarray*}
 \dim {\rm Im}\Big( \alpha: C_1(\mathcal{K};\mathbb{R})\longrightarrow C_0(\mathcal{K};\mathbb{R})\Big)=  \begin{cases}
 2,  & {\rm~ if ~} f(v_i)\neq  0{\rm ~for ~some~}i=0,1,2;\\
 0,  &{\rm ~if ~} f(v_0)=f(v_1)=f(v_2)=0.
  \end{cases}
   \end{eqnarray*}
  Thus
  \begin{eqnarray*}
  H_0(\mathcal{K},f)=H_0(\mathcal{K},\alpha,0)= \begin{cases}
 \mathbb{R},  & {\rm~ if ~} f(v_i)\neq  0{\rm ~for ~some~}i=0,1,2;\\
 \mathbb{R}^3,  &{\rm ~if ~} f(v_0)=f(v_1)=f(v_2)=0.
  \end{cases}
  \end{eqnarray*}
  Note that
   \begin{eqnarray*}
 \dim {\rm Ker}\Big(\alpha : C_1(\mathcal{K};\mathbb{R})\longrightarrow C_0(\mathcal{K};\mathbb{R})\Big)=
 \begin{cases}
 1,   & {\rm~ if ~} f(v_i)\neq  0{\rm ~for ~some~}i=0,1,2;\\
 3,  &{\rm ~if ~} f(v_0)=f(v_1)=f(v_2)=0
  \end{cases}
 \end{eqnarray*}
and
\begin{eqnarray*}
 \dim {\rm Im}\Big( \alpha: C_2(\mathcal{K};\mathbb{R})\longrightarrow C_1(\mathcal{K};\mathbb{R})\Big)=0.
   \end{eqnarray*}
 Thus
  \begin{eqnarray*}
H_1(\mathcal{K},f)=  H_1(\mathcal{K},\alpha,0)= \begin{cases}
 \mathbb{R},  & {\rm~ if ~} f(v_i)\neq  0{\rm ~for ~some~}i=0,1,2;\\
 \mathbb{R}^3,  &{\rm ~if ~} f(v_0)=f(v_1)=f(v_2)=0.
  \end{cases}
  \end{eqnarray*}
  \item
Let $s=1$.   Let
\begin{eqnarray*}
\beta=b_{01} \frac{\partial}{\partial v_0}\wedge  \frac{\partial}{\partial v_1} + b_{02} \frac{\partial}{\partial v_0}\wedge  \frac{\partial}{\partial v_2} +b_{12} \frac{\partial}{\partial v_1}\wedge  \frac{\partial}{\partial v_2}.
\end{eqnarray*}
Then
\begin{eqnarray*}
\beta(v_i)=0 ~~~~~~{\rm ~for~} 0\leq  i\leq  2
\end{eqnarray*}
and
\begin{eqnarray*}
\beta(v_iv_j)=0 ~~~~~~{\rm ~for~} 0\leq  i<j\leq  2.
\end{eqnarray*}
Thus the induced homomorphism $\beta_*$  between the homology groups is identically zero.
\end{itemize}

\item
 Let
\begin{eqnarray*}
\mathcal{L}=\{v_0v_1, v_0v_2, v_0v_1v_2\}
\end{eqnarray*}
  be a independent hypergraph  with its vertices in $V$. Then we   have
  \begin{eqnarray*}
  C_0(\mathcal{L};\mathbb{R})&=&0, \\
C_1(\mathcal{L};\mathbb{R})&=& {\rm Span}_{\mathbb{R}}\{v_0v_1, v_0v_2\},\\
        C_2(\mathcal{L};\mathbb{R})&=&{\rm Span}_{\mathbb{R}}\{v_0v_1v_2\},\\
         C_n(\mathcal{L};\mathbb{R})&=& 0 {\rm~~~~ for ~all~ } n\geq 3.
  \end{eqnarray*}
  \begin{itemize}
  \item
   Let $t=1$.  Let
   \begin{eqnarray*}
   \omega&=&\sum_{v\in V}  f(v) dv\\
   &=& f(v_0) dv_0 + f(v_1) dv_1  + f(v_2) dv_2+  f(v_3)  dv_3.
   \end{eqnarray*}
    With the help of  Example~\ref{ex-4.3.1}~(ii),   we  have
 \begin{eqnarray*}
 &\omega (v_0v_1)= f(v_2) v_0v_1v_2,\\
  &\omega (v_0v_2)= -f(v_1) v_0v_1v_2,\\
 &\omega (v_0v_1v_2)= 0.
 \end{eqnarray*}
 Note that
 \begin{eqnarray*}
 \dim {\rm Ker}\Big(\omega :  C_1(\mathcal{L};\mathbb{R})\longrightarrow C_2(\mathcal{L};\mathbb{R})\Big)=
 \begin{cases}
 1,  & {\rm~ if ~} f(v_i)\neq  0{\rm ~for ~some~}i=1,2;\\
 2,  &{\rm ~if ~} f(v_1)=f(v_2)=0
  \end{cases}
 \end{eqnarray*}
 or equivalently,
  \begin{eqnarray*}
 \dim {\rm Im}\Big(\omega :  C_1(\mathcal{L};\mathbb{R})\longrightarrow C_2(\mathcal{L};\mathbb{R})\Big)=
 \begin{cases}
 1,  & {\rm~ if ~} f(v_i)\neq  0{\rm ~for ~some~}i=1,2;\\
 0,  &{\rm ~if ~} f(v_1)=f(v_2)=0.
  \end{cases}
 \end{eqnarray*}
 Thus
 \begin{eqnarray*}
 H^1(\mathcal{L},f)=H^1(\mathcal{L},\omega,0)=\begin{cases}
 \mathbb{R},  & {\rm~ if ~} f(v_i)\neq  0{\rm ~for ~some~}i=1,2;\\
 \mathbb{R}^2,  &{\rm ~if ~} f(v_1)=f(v_2)=0
  \end{cases}
  \end{eqnarray*}
  and
  \begin{eqnarray*}
 &H^2(\mathcal{L},f)=H^2(\mathcal{L},\omega,0)=\begin{cases}
 0,  & {\rm~ if ~} f(v_i)\neq  0{\rm ~for ~some~}i=1,2;\\
 \mathbb{R},  &{\rm ~if ~} f(v_1)=f(v_2)=0.
  \end{cases}
 \end{eqnarray*}
 Moreover,
  \begin{eqnarray*}
 &H^n(\mathcal{L},f)=H^n(\mathcal{L},\omega,0)= 0
 \end{eqnarray*}
 for any $n\neq  1,2$.

  \item
Let $s=1$.   Let
\begin{eqnarray*}
\mu=u_{01} d v_0 \wedge  d  v_1  + u_{02} d v_0 \wedge  d v_2 +u_{12} d v_1 \wedge  d v_2.
\end{eqnarray*}
Then
\begin{eqnarray*}
\mu(v_0v_1v_2)=mu(v_0v_1)=\mu(v_0v_2)=0.
\end{eqnarray*}
Thus the induced homomorphism $\mu_*$  between the cohomology groups is identically zero.
   \end{itemize}
    \end{enumerate}
\end{example}

\begin{example}
Let $V=\{v_0,v_1,v_2,v_3\}$.
   Let $t=1$.  Then any
 $\alpha \in {\rm Ext}_3(V)$  can be expressed  as
\begin{eqnarray*}
\alpha&=&  f(v_0,v_1,v_2)\frac{\partial}{\partial v_0}\wedge\frac{\partial }{\partial v_1}\wedge\frac{\partial}{\partial v_2}
+ f(v_0,v_1,v_3)\frac{\partial}{\partial v_0}\wedge\frac{\partial }{\partial v_1}\wedge\frac{\partial}{\partial v_3}\\
&&
 + f(v_0,v_2,v_3)\frac{\partial}{\partial v_0}\wedge\frac{\partial }{\partial v_2}\wedge\frac{\partial}{\partial v_3}
 + f(v_1,v_2,v_3)\frac{\partial}{\partial v_1}\wedge\frac{\partial }{\partial v_2}\wedge\frac{\partial}{\partial v_3}
\end{eqnarray*}
where
\begin{eqnarray*}
f: ~~~ V\times V\times V\longrightarrow \mathbb{R}
\end{eqnarray*}
  is a real function on the $3$-fold Cartesian product of $V$.  By   Proposition~\ref{le-2.1},  the adjoint $\omega\in {\rm Ext}^3(V)$ of $\alpha$  is  given by
\begin{eqnarray*}
\omega&=&-  f(v_0,v_1,v_2)dv_0 \wedge dv_1\wedge dv_2
- f(v_0,v_1,v_3) dv_0\wedge dv_1\wedge dv_3\\
&&
 - f(v_0,v_2,v_3)dv_0\wedge dv_2\wedge dv_3
 - f(v_1,v_2,v_3)dv_1\wedge dv_2\wedge dv_3.
\end{eqnarray*}
Let $s=1$.  Then any
 $\beta \in {\rm Ext}_2(V)$  can be expressed  as
\begin{eqnarray*}
\beta&=&  g(v_0,v_1)\frac{\partial}{\partial v_0}\wedge\frac{\partial }{\partial v_1}
+ g(v_0,v_2)\frac{\partial}{\partial v_0}\wedge\frac{\partial }{\partial v_2}+ g(v_0,v_3)\frac{\partial}{\partial v_0}\wedge\frac{\partial }{\partial v_3}\\
&&
+  g(v_1,v_2)\frac{\partial}{\partial v_1}\wedge\frac{\partial }{\partial v_2}
+ g(v_1,v_3)\frac{\partial}{\partial v_1}\wedge\frac{\partial }{\partial v_3}+ g(v_2,v_3)\frac{\partial}{\partial v_2}\wedge\frac{\partial }{\partial v_3}
\end{eqnarray*}
where
\begin{eqnarray*}
g: ~~~ V\times V \longrightarrow \mathbb{R}
\end{eqnarray*}
  is a real function on the $2$-fold Cartesian product of $V$.  By   Proposition~\ref{le-2.1},  the adjoint $\mu\in {\rm Ext}^2(V)$ of $\beta$  is  given by
\begin{eqnarray*}
\mu&=&-  g(v_0,v_1)  d v_0 \wedge   d v_1
-g(v_0,v_2) d  v_0 \wedge   d  v_2 - g(v_0,v_3) d v_0 \wedge   d v_3 \\
&&
-  g(v_1,v_2)  d v_1 \wedge   d v_2
- g(v_1,v_3)   d v_1 \wedge   d v_3 - g(v_2,v_3)  d v_2 \wedge  d v_3.
\end{eqnarray*}
Consider the   complete hypergraph
   \begin{eqnarray*}
   \Delta[V]&=&\{v_0,v_1,v_2,v_3, v_0v_1, v_0v_2,v_0v_3,v_1v_2,v_1v_3,v_2v_3, \\
   &&v_0v_1v_2, v_0v_1v_3, v_0v_2v_3, v_1v_2v_3, v_0v_1v_2v_3\}.
   \end{eqnarray*}
     Then  $\Delta[V]$  is    a  simplicial complex  and is also a independent hypergraph.
     \begin{itemize}
     \item
     By a  direct calculation,
\begin{eqnarray*}
\alpha (v_i)&=&0, ~~~  i=0,1,2,3,\\
 \alpha(v_i v_j)&=&0,  ~~~0\leq i<j\leq 3,\\
 \alpha(v_i v_j v_k)&=&0,~~~ 0\leq i<j<k\leq 3,\\
 \alpha(v_0v_1v_2v_3)&=& (-1)^{0+1+2} f(v_0,v_1,v_2) v_3 + (-1)^{0+1+3} f(v_0,v_1,v_3) v_2\\
 &&  + (-1)^{0+2+3}  f(v_0,v_2,v_3)  v_1 + (-1)^{1+2+3} f(v_1,v_2,v_3) v_0\\
 &=&- f(v_0,v_1,v_2) v_3 +   f(v_0,v_1,v_3) v_2
 -  f(v_0,v_2,v_3)  v_1  \\
 &&  +   f(v_1,v_2,v_3) v_0.
\end{eqnarray*}
It follows that
\begin{eqnarray*}
&&\dim {\rm Im}\Big(\alpha :  C_3(\Delta[V];\mathbb{R})\longrightarrow C_0(\Delta[V];\mathbb{R})\Big)\\
&=&
\begin{cases}
1, & {\rm~if~}f(v_i,v_j,v_k), 0\leq i<j<k\leq 3, {\rm~are ~not~all~zero}; \\
0, & {\rm~if~}f(v_i,v_j,v_k)=0 {\rm~for ~any~} 0\leq i<j<k\leq 3
\end{cases}
\end{eqnarray*}
or equivalently,
\begin{eqnarray*}
&&\dim {\rm Ker}\Big(\alpha :  C_3(\Delta[V];\mathbb{R})\longrightarrow C_0(\Delta[V];\mathbb{R})\Big)\\
&=&
\begin{cases}
0, & {\rm~if~}f(v_i,v_j,v_k), 0\leq i<j<k\leq 3, {\rm~are ~not~all~zero}; \\
1, & {\rm~if~}f(v_i,v_j,v_k)=0 {\rm~for ~any~} 0\leq i<j<k\leq 3.
\end{cases}
\end{eqnarray*}
 Consequently,
\begin{eqnarray*}
H_0(\Delta[V],\alpha,0)=
 \begin{cases}
3, & {\rm~if~}f(v_i,v_j,v_k), 0\leq i<j<k\leq 3, {\rm~are ~not~all~zero}; \\
4, & {\rm~if~}f(v_i,v_j,v_k)=0 {\rm~for ~any~} 0\leq i<j<k\leq 3
\end{cases}
\end{eqnarray*}
and
\begin{eqnarray*}
H_3(\Delta[V],\alpha,0)=
 \begin{cases}
0, & {\rm~if~}f(v_i,v_j,v_k), 0\leq i<j<k\leq 3, {\rm~are ~not~all~zero}; \\
1, & {\rm~if~}f(v_i,v_j,v_k)=0 {\rm~for ~any~} 0\leq i<j<k\leq 3.
\end{cases}
\end{eqnarray*}
By a similar calculation,   we  have
\begin{eqnarray*}
H_1(\Delta[V],\alpha,0)=\mathbb{R}^6,~~~~~~ H_2(\Delta[V],\alpha,0)=\mathbb{R}^4.
\end{eqnarray*}
Moreover,
\begin{eqnarray*}
H_n(\Delta[V],\alpha,0)=0
\end{eqnarray*}
for any $n\neq 0,1,2,3$.

\item
It is direct that
\begin{eqnarray*}
\beta\circ\alpha(v_i)=0
\end{eqnarray*}
for any $0\leq  i\leq 3$,
\begin{eqnarray*}
\beta\circ\alpha(v_iv_j)=0
\end{eqnarray*}
for any $0\leq  i<j\leq 3$,
\begin{eqnarray*}
\beta\circ\alpha(v_iv_jv_k)=0
\end{eqnarray*}
for any $0\leq  i<j<k\leq 3$,
and
\begin{eqnarray*}
\beta\circ\alpha(v_0v_1v_2v_3)=0.
\end{eqnarray*}
Therefore,  the  induced homomorphism $\beta_*$  between the homology groups  is the zero map.

\item
By  a direct calculation,
  \begin{eqnarray*}
  \omega(v_0)&=& -f(v_1,v_2,v_3) dv_1\wedge dv_2\wedge dv_3 (v_0)\\
  &=& f(v_1,v_2,v_3) v_0v_1v_2v_3,\\
  \omega(v_1)&=& -f(v_0,v_2,v_3) dv_0\wedge dv_2\wedge dv_3 (v_1)\\
  &=& -f(v_0,v_2,v_3)  v_0v_1v_2v_3,\\
  \omega(v_2)&=&- f(v_0,v_1,v_3) dv_0\wedge dv_1\wedge dv_3 (v_2)\\
  &=& f(v_0,v_1,v_3)  v_0v_1v_2v_3,\\
  \omega(v_3)&=&- f(v_0,v_1,v_2) dv_0\wedge dv_1\wedge dv_2 (v_3)\\
  &=& -f(v_0,v_1,v_2)  v_0v_1v_2v_3, \\
  \omega(v_iv_j)&=&0,~~~ 0\leq i<j\leq 3,\\
    \omega(v_iv_jv_k)&=&0,~~~ 0\leq i<j<k\leq 3,\\
    \omega(v_0v_1v_2v_3)&=&0.
   \end{eqnarray*}
   It follows that
\begin{eqnarray*}
&&\dim {\rm Im}\Big(\omega :  C_0(\Delta[V];\mathbb{R})\longrightarrow C_3(\Delta[V];\mathbb{R})\Big)\\
&=&
\begin{cases}
1, & {\rm~if~}f(v_i,v_j,v_k), 0\leq i<j<k\leq 3, {\rm~are ~not~all~zero}; \\
0, & {\rm~if~}f(v_i,v_j,v_k)=0 {\rm~for ~any~} 0\leq i<j<k\leq 3
\end{cases}
\end{eqnarray*}
or equivalently,
\begin{eqnarray*}
&&\dim {\rm Ker}\Big(\omega :  C_0(\Delta[V];\mathbb{R})\longrightarrow C_3(\Delta[V];\mathbb{R})\Big)\\
&=&
\begin{cases}
3, & {\rm~if~}f(v_i,v_j,v_k), 0\leq i<j<k\leq 3, {\rm~are ~not~all~zero}; \\
4, & {\rm~if~}f(v_i,v_j,v_k)=0 {\rm~for ~any~} 0\leq i<j<k\leq 3.
\end{cases}
\end{eqnarray*}
Consequently,
\begin{eqnarray*}
H^0(\Delta[V],\omega,0)=
 \begin{cases}
3, & {\rm~if~}f(v_i,v_j,v_k), 0\leq i<j<k\leq 3, {\rm~are ~not~all~zero}; \\
4, & {\rm~if~}f(v_i,v_j,v_k)=0 {\rm~for ~any~} 0\leq i<j<k\leq 3
\end{cases}
\end{eqnarray*}
and
\begin{eqnarray*}
H^3(\Delta[V],\omega,0)=
 \begin{cases}
0, & {\rm~if~}f(v_i,v_j,v_k), 0\leq i<j<k\leq 3, {\rm~are ~not~all~zero}; \\
1, & {\rm~if~}f(v_i,v_j,v_k)=0 {\rm~for ~any~} 0\leq i<j<k\leq 3.
\end{cases}
\end{eqnarray*}
By a similar calculation,   we  have
\begin{eqnarray*}
H^1(\Delta[V],\omega,0)=\mathbb{R}^6,~~~~~~ H^2(\Delta[V],\omega,0)=\mathbb{R}^4.
\end{eqnarray*}
Moreover,
\begin{eqnarray*}
H^n(\Delta[V],\omega,0)=0
\end{eqnarray*}
for any $n\neq 0,1,2,3$.

\item
   It is direct to see that  the  induced homomorphism $\mu_*$  between the cohomology groups  is the zero map.
   \end{itemize}
\end{example}

 \smallskip

\section*{Acknowledgement} {The  author would like to express his  deep
gratitude to the referee for the careful reading of the manuscript.}

 \smallskip
 {\small
Shiquan Ren

Address:   School of Mathematics and Statistics,  Henan University,  Kaifeng  475004,  China.

E-mail:  renshiquan@henu.edu.cn}

\end{document}